\documentclass[11pt]{amsart}
\usepackage{amsmath,amssymb,amscd,mathrsfs}
\usepackage{graphicx}
\usepackage{pict2e}
\newtheorem{introth}{Theorem}
\newtheorem{introex}[introth]{Example}
\newtheorem{introdefinition}[introth]{Definition}
\newenvironment{introdfn}{\begin{introdefinition}\rm}{\end{introdefinition}}
\newtheorem{introremark}[introth]{Remark}

\newtheorem{thm}{Theorem}[section]
\newtheorem{prop}[thm]{Proposition}
\newtheorem{lem}[thm]{Lemma}
\newtheorem{cor}[thm]{Corollary}
\newtheorem{definition}[thm]{Definition} 
\newtheorem{remark}[thm]{Remark}
\newtheorem{observation}[thm]{Observation}
\newtheorem{example}[thm]{Example}

\newenvironment{dfn}{\begin{definition}\rm}{\end{definition}}
\newenvironment{rmk}{\begin{remark}\rm}{\end{remark}}

\newcommand{\act}{\curvearrowright}

\def\Aut{\mathop{\mathrm{Aut}}\nolimits}
\newcommand{\bracket}[1]{{\langle {#1} \rangle}}

\def\diag{\mathop{\mathrm{diag}}\nolimits}
\def\e{\chi_{\mathrm{top}}}
\def\Fix{\mathop{\mathrm{Fix}}\nolimits}
\def\GL{\mathop{\mathrm{GL}}\nolimits}
\def\id{\mathop{\mathrm{id}}\nolimits}
\def\ord{\mathop{\mathrm{ord}}\nolimits}
\def\tr{\mathop{\mathrm{tr}}\nolimits}
\def\Out{\mathop{\mathrm{Out}}\nolimits}
\def\splitext{\! :\!}

\def\C{\mathbb C} \def\Q{\mathbb Q} \def\R{\mathbb R} \def\Z{\mathbb Z}
\def\P{\mathbb P} \def\F{\mathbb F}

\newenvironment{(enumerate)}{
  \begin{enumerate}
  
  }{\end{enumerate}}
\newenvironment{anumerate}{
  \begin{enumerate}
  
  }{\end{enumerate}}

\numberwithin{equation}{section}

\begin{document} 
\title[Finite groups of automorphisms of Enriques surfaces]{Finite groups of automorphisms of Enriques surfaces and the Mathieu group  $M_{12}$}
\author[S. Mukai]{Shigeru Mukai}
\author[H. Ohashi]{Hisanori Ohashi}
\address{Research Institute for Mathematical Sciences,
Kyoto University,
Kyoto 606-8502,
JAPAN }
\email{mukai@kurims.kyoto-u.ac.jp}
\address{Department of Mathematics, 
Faculty of Science and Technology, 
Tokyo University of Science, 
2641 Yamazaki, Noda, 
Chiba 278-8510, JAPAN}
\email{ohashi@ma.noda.tus.ac.jp, ohashi.hisanori@gmail.com}
\thanks{Supported in part by 
the JSPS Grant-in-Aid for Scientific Research (B) 22340007, (S) 19104001, (S) 22224001, (A) 22244003, (S)25220701, for Exploratory Research 20654004 and for Young Scientists (B) 23740010.}

\begin{abstract}
An action of a group $G$ on an Enriques surface $S$ is called Mathieu if  it acts on $H^0(2K_S)$ trivially and
every element
of order  2, 4 has Lefschetz number  4.
A finite group $G$ has a Mathieu action on some Enriques surface if and only if it is isomorphic to a subgroup
of the symmetric group $\mathfrak S_6$ of degree 6 and the order  $|G|$  is not divisible by $2^4$.
Explicit Mathieu actions of the three groups $\mathfrak S_5, N_{72}$ and $\mathfrak A_6$,
together with non-Mathieu one of $H_{192}$, on polarized Enriques surfaces of degree 30, 18, 10 and 6, respectively, are constructed without Torelli type theorem to prove the ‘if’ part.
%The \lq if\rq\,  part is proved constructing explicit actions of  three groups $\mathfrak S_5, N_{72}$ and $H_{192}$  on polarized Enriques surfaces of degree 30, 18 and 6, respectively 
%and using the result by Keum--Oguiso--Zhang for the alternating group $\mathfrak A_6$.
%the Mathieu group  $M_{12} $ on 
%the Niemeier lattice of type  $A_1^{24}$ and a K3 surface of degree 12 with a free involution and 192 symplectic automorphisms.
%the sextic Enriques surface 
%$\sum_1^4 x_i^2+ \sqrt{-1}x_1x_2x_3x_4\sum_1^4 x_i^{-2}=0$
%in  $\P^3$.
\end{abstract}
\date{\today} 
\maketitle

A (holomorphic) action of  a group on a $K3$ surface  $X$  is {\it symplectic} if it acts on $H^0(K_X) \simeq \C$  trivially.
The finite groups which can act symplectically on $K3$ surfaces  are classified in \cite{mukai88}, relating with the Mathieu group  $M_{23}$.
There are exactly eleven maximal groups
\begin{equation}\tag{$\ast$}%\label{11}
L_2(7),\ \mathfrak A_6,\ \mathfrak S_5,\ M_{20},\ F_{384},\ \mathfrak A_{4,4},\ T_{192},\ H_{192},\ N_{72},\ M_9,\ T_{48}
\end{equation}
among them.
In this article we give a similar classification for Enriques surfaces, relating with the symmetric group $\mathfrak S_6$ of degree 6 embedded in the Mathieu group  $M_{12}$.

A (minimal) Enriques surface $S$  is a smooth complete algebraic surface with  
$q=h^1(\mathcal O_S)=0,\ p_g=h^0(K_S)=0$
and  $2K_S \sim 0$.
Equivalently, $S$ is the quotient of a $K3$ surface  $X$  by a (fixed point) free involution  $\varepsilon$. 
An action of  a group on an Enriques surface  $S$  is {\it semi-symplectic}  if it acts on $H^0(2K_S) \simeq \C$  trivially.
If a finite automorphism $\sigma$  is semi-symplectic, then we have $\ord(\sigma) \le 6$
(Corollary \ref{atmost6}) and the Lefschetz number
%\footnote{See \eqref{lefschetz}, Section \ref{preliminary} for the definition.}
$L(\sigma)$ takes the following values (Proposition \ref{chara}).
\begin{equation*}%\label{}
\begin{tabular}{c||c|c|c|c|c|c}
$\mathrm{ord}(\sigma)$ &1&2&3&4&5&6\\\hline
$L(\sigma)$ &12& $-4, -2$, \ldots, 12&3&2, 4&2&1, 3
\end{tabular}
\end{equation*}
Although the Lefschetz number of a finite symplectic automorphism depends only on its order
for $K3$ surfaces, the above table disproves
the similar statement for semi-symplectic automorphisms of Enriques surfaces
(see also Example~\ref{H192}, Remark~\ref{Clebsch} and Remark~\ref{far from Mathieu}).
In this paper we study the most natural subcases from the viewpoint of the small Mathieu group $M_{12}$.
A classification of general semi-symplectic actions will be discussed elsewhere.

The Mathieu group $M_{12}$  is a finite simple group of sporadic type.
It acts on  $\Omega_+= \{1, \ldots, 12\}$  quintuply transitively and is of order $12\cdot 11\cdot 10\cdot 9\cdot 8=2^6\cdot3^3\cdot5\cdot11$.
The stabilizer subgroup of  $M_{12}$  at a point $\star \in \Omega_+$ is denoted by $M_{11}$.
The number of fixed points of the permutation action $M_{11} \act \Omega_+$  depends only on the order of an element and is given as follows.
\begin{equation}\tag{$\ast\ast$}%\label{M11}
\begin{tabular}{c|cccccccc}
{\rm order} &1&2&3&4&5&6&8&11\\\hline
{\rm number of fixed points} &12&4&3&4&2&1&2&1
\end{tabular}
\end{equation}

%By the fixed point formula, the Lefschetz number $\chi_{top}({\rm Fix}(g))$, $g \in G$, of the action  $G \act S$  of the theorem is given also by ($\ast\ast$).
After Table~($\ast\ast$), we make the following

\begin{introdfn}\label{Mathieu}
A semi-symplectic action of a group  $G$  on an Enriques surface  is {\it Mathieu} if the 
Lefschetz number of $g\in G$ depends only on $\ord (g)$ and coincides with 
the lower row in table ($\ast\ast$).
\end{introdfn}
For an action of $G$ to be Mathieu, 
it suffices for elements $g\in G$ of order $2$ or $4$ to have 
Lefschetz number $4$ by Lemma \ref{criterion}.
Mathieu actions are automatically effective by definition. Our main result
is as follows.
\begin{introth}\label{main}
For a finite group  $G$, the following two conditions are equivalent to each other.
\begin{(enumerate)}
\item $G$  has a Mathieu action on some Enriques surface.
\item $G$  can be embedded into the symmetric group $\mathfrak S_6$  and the order  $|G|$  is not divisible by $2^4$.
\end{(enumerate)}
\end{introth}

The following examples are the key of the proof of  (2) $\Rightarrow$ (1).

\begin{introex}\label{Ohashi}
Let  $X$ be the minimal resolution of the complete intersection
$$\sum_{i<j} x_ix_j = \sum_{i<j} \frac1{x_ix_j} =0$$  
in  $\P^4$ and  $S$ its quotient by the involution induced by 
the Cremona transformation $\varepsilon: (x_i) \mapsto (1/x_i)$.
Then the natural action of  $\mathfrak S_5$, the third group of ($\ast$),  on  $S$ is (semi-symplectic and) Mathieu (\S\S\ref{subsection:S5}). (In fact, this Enriques surface 
is isomorphic to the surface of type VII in \cite{kondo86}. See Remark \ref{kondo7}.) 
\end{introex}

\begin{introex}\label{Mukai}
Let $S=X/\varepsilon$ be the quotient of the complete intersection
$$X: x_i^2 - (1+\sqrt{3})x_{i-1}x_{i+1} = y_i^2 - (1-\sqrt{3})y_{i-1}y_{i+1}, \quad i =0,1,2\in \Z/3$$
in $\P^5$ by the involution $\varepsilon: (x: y) \mapsto (x: -y)$.
A subgroup $C_3^2\splitext C_4$ of the Hessian group $G_{216}$ 
acts on  $X$  linearly and also on $S$.
This action extends to that of $N_{72}$, the ninth of ($\ast$).
The extended action is Mathieu (\S\S\ref{subsection:N72}).
\end{introex}

%The difference of the 1st and 2nd defining equations of $X$ reads as
%\begin{equation}\label{rank4}
%(x_0-x_1)(x_0+x_1+\lambda x_2)=(y_0-y_1)(y_0+y_1+\mu y_2).
%\end{equation}
%Hence 
The rational functions
$F_{00}=(y_0-y_1)/(x_0-x_1)$
and its square define elliptic fibrations  $X \to \P^1$   and $S \to \P^1$   of the $K3$ and Enriques surface in the above example, respectively.
The elliptic fibration $S \to \P^1$ has four singular fibers of type $I_3$.
Hence its Jacobian fibration is induced by the Hesse pencil and $S$ has a semi-symplectic action of  $C_3^2$, the Mordell-Weil group.  % of the Jacobian fibration.

\begin{introex}\label{KOZex}
The action of $C_3^2\splitext C_4$  in Example~\ref{Mukai} and the above $C_3^2$ generate a Mathieu action of $\mathfrak{A}_6$,  the second of ($\ast$), on the Enriques surface $S$  (\S\S\ref{subsection:A6}).
%This surface has an Enriques quotient and the induced action by 
%$\mathfrak{A}_6$, the second of ($\ast$). This action is Mathieu.
\end{introex}

%\begin{intrormk}
%In fact, we can show that the Enriques surfaces in Examples \ref{Mukai}, \ref{KOZex} are 
%isomorphic to each other. Thus $S$ has both actions by $\mathfrak{A}_6$ and 
%$N_{72}$. The details will be published elsewhere (cf. Remark~\ref{N72&A6}).
%\end{intrormk}

The eighth group  $H_{192}$ of ($\ast$), not a subgroup of  $M_{11}$, has the following action.

\begin{introex}\label{H192}
The surface
\begin{equation*}
X : v^2w^2 + u^2w^2 + u^2v^2 +1 + \sqrt{-1}(u^2 + v^2 + w^2 + u^2v^2w^2) = 0 
\end{equation*}
of tri-degree $(2,2,2)$  in  $\P^1 \times \P^1 \times \P^1$  is a smooth $K3$ surface with a symplectic action of the group $H_{192}$,
where $u, v, w$  are the inhomogeneous coordinates of three projective lines.
The involution $\varepsilon: (u,v,w) \mapsto (-u,-v,-w)$  is free, commutes with the action and hence the quotient Enriques surface  $S = X/\varepsilon$  has a semi-symplectic action of $H_{192}$.
\end{introex}
The involutions induced by
$$(u, v,w) \mapsto (u, -v, -w) \quad  {\rm and} \quad (u, v,w) \mapsto (-\sqrt{-1}/u, -\sqrt{-1}/v,-\sqrt{-1}/w)$$
are Mathieu,
but  $(u, v,w) \mapsto (u, w,v)$  is not.
 (In fact the Lefschetz number equals  2 on $S$.)
%After ($\ast\ast$), we make the following
%
%\begin{introdfn}\label{Mathieu}
%An (effective) semi-symplectic action of a group  $G$  on an Enriques surface  is {\it Mathieu} if the Lefschetz number of  $g$  equals $4$  for every  $g \in G$ of order  $2$ and $4$.
%\end{introdfn}
%Obviously three actions of Theorem~\ref{sub-main}  is Mathieu.
Thus the action of Example~\ref{H192} is not Mathieu.
But we can find two Mathieu sub-actions by  $C_2 \times \mathfrak A_4$ and $C_2\times C_4$,
see \cite{Fields} and also Section \ref{3->1}.
Though neither is a subgroup of  $M_{11}$, both are subgroups of  $\mathfrak S_6$,
which is why the group $\mathfrak{S}_6$ appears in Theorem \ref{main}. 
%We have the following characterization in the latter,
%We denote the subset of non-free permutations of  $M_{12} \act \Omega_+$ by  $M'_{12}$.
%which is our main result.

\begin{introth}\label{main2}
The two conditions in Theorem~\ref{main} are equivalent to the following:
\begin{(enumerate)}
\addtocounter{enumi}{2}
%\item $G$  has a Mathieu action on an Enriques surface.
\item $G$  is a subgroup of one of the five maximal groups $\mathfrak A_6, \mathfrak S_5, N_{72}, C_2 \times \mathfrak A_4$ and $C_2\times C_4$.
%\item $G$  has an embedding into the subset of non-free permutations $M'_{12} \subset M_{12}$ such that $G$  decomposes $\Omega_+$  into at least three orbits and  $\Omega_- := \Omega \setminus \Omega_+$  into at least two orbits.
%\item $G$  can be embedded into the symmetric group $\mathfrak S_6$  and  the order  $|G|$  is not divisible by $2^4$.
\item $G$  has a small Mathieu representation (Definition \ref{Maction} (1)) with  $\dim V^G$ $\ge 3$,
its $2$-Sylow subgroup is embeddable into  $\mathfrak S_6$ and  $G \not\simeq Q_{12}$, the generalized quaternion group of order $12$.
\end{(enumerate)}
\end{introth}
Finally using these results in characteristic zero, we extend our classification
%the classification of groups with Mathieu actions 
to tame Mathieu actions in positive characteristic.  %in Section \ref{pc}.

\begin{introth}\label{main3}
Let $k$ be an algebraically closed field of positive characteristic $p>0$. 
For a finite group  $G$ with  $(|G|, p)=1$, the two conditions in Theorem~\ref{main} are equivalent, that is,
$G$  has a Mathieu action on some Enriques surface over $k$ if and only if the condition (2) holds.
\end{introth}
%See Theorem \ref{positive_char} and Proposition \ref{positive_char2}.

%\begin{intrormk}
%(1) Though not a subgroup of  $M_{11}$, $C_2 \times \mathfrak A_4 \subset M_{12}$ consists of permutations which are not free on $\Omega_+$.
%$C_2 \times C_4 \subset M_{12}$ consists of permutations not free on $\Omega$.
%
%(2) The last group $T_{48} = Q_8\cdot\mathfrak S _3$ is a subgroup of $M_{11}$ and decomposes both $\Omega+\setminus \{\star\}$  and  $\Omega_-$  into two orbits like the three groups in Theorem~\ref{sub-main}.
%But $T_{48}$  does not have a Mathieu action on an Enriques surface.
%\end{intrormk}
%
%After a preliminary consideration in \S\ref{preliminary}, we prove  $(2) \Rightarrow (1)$  in \S\ref{ltc}  and its converse in \S\ref{classification}.
%Theorem~\ref{sub-main}, the main part of $(2) \Rightarrow (1)$, is proved in \S\ref{ltc} by refining the argument of \cite[Appendix]{mukai98}.
%The equivalence of  (2), (3)  and  (4)  are proved in \S\ref{equivalence}.
%
The construction of the paper is as follows. 
We prove $(3) \Rightarrow (1)$ of Theorems~\ref{main} and \ref{main2} for the 
three groups $\mathfrak{S}_5$, $N_{72}$, $\mathfrak{A}_6$ in a refined form (Theorem~\ref{veryMathieu}) in Sections \ref{examples} and \ref{section:Hesse-Godeaux}, where we intensively study 
the surfaces in Examples \ref{Ohashi}, \ref{Mukai}, \ref{KOZex}.
Mathieu actions for the other groups $C_2\times \mathfrak{A}_4$ and $C_2\times C_4$
are constructed in Section \ref{3->1}. 
They are nothing but the actions studied in detail in \cite{Fields}, but here we give a slightly different treatment.

We give a preliminary study of semi-symplectic and Mathieu automorphisms in
Section \ref{preliminary}. 
%Its main part, Theorem~\ref{sub-main}, is proved there by refining the argument of \cite[Appendix]{mukai98}. 
In Section \ref{sMrep} we study groups with small Mathieu representations.
Finally in Section \ref{equivalence} we prove the other implications of Theorem~\ref{main} and  Theorem~\ref{main2} and complete the proofs.
Especially in Subsection \ref{4to2} we classify all finite groups satisfying the equivalent conditions of main theorems. 
In Section \ref{pc}, we 
%extend our classification to tame groups in positive characteristic and 
prove Theorem~\ref{main3}.

In Appendix \ref{ltc}, an Enriques analogue of \cite[Appendix]{mukai98}, from which this article stems, is presented to give an alternative proof of Theorem~\ref{veryMathieu}.
In Appendix \ref{KOZconstruction}, yet another lattice-theoretic construction of an
$\mathfrak{A}_6$ action using the result of \cite{KOZ1, KOZ2} is presented. 

\medskip
\noindent{\bf Notation and conventions.}

Algebraic varieties $X$ are considered over the complex number field $\C$, 
except Theorem \ref{main3} and Section \ref{pc} where it is over an algebraically closed field $k$ of positive characteristic.

For a smooth variety $X$, $K_X$ denotes the canonical divisor class.

Notation of finite groups follows \cite{mukai88}. In particular, 
$C_n$, $D_{2n}$, $Q_{4n}$, $\mathfrak S_n$, $\mathfrak A_n$ denotes 
the cyclic group of order $n$, the dihedral group of order $2n$, 
the generalized quaternion group of order $4n$, the symmetric  group of degree $n$, the alternating group of degree $n$ respectively. The definition of some other groups will be 
recalled when it is necessary.
For groups $A$ and $B$, $A\splitext B$ denotes a split extension with normal subgroup $A$. 

The Mathieu group $M_{24}$ acts on the operator domain $\Omega$ consisting of $24$ points
and the small Mathieu group $M_{12}$ is the stabilizer of a dodecad, denoted by $\Omega_+$.

The symbol $U$ denotes the rank $2$ lattice given by the symmetric matrix $\begin{pmatrix} 0&1 \\ 1&0 \end{pmatrix}$. Root lattices $A_n, D_n$ and $E_n$ are considered 
negative definite.
The lattice obtained from a lattice $L$ by replacing the bilinear form $(\ .\ )$ with $r(\ .\ )$, $r$ being a rational number, is denoted by $L(r)$.

\section{The first example of Mathieu action}\label{examples}
We recall that the symmetric group  $\mathfrak S_6$  is a subgroup of  $M_{24}$  and decomposes the operator domain  $\Omega$  into four orbits  $\Omega_2, \Omega_{10}, \Omega_6, \Omega_6'$ of length  $2, 10, 6$ and 6 (\cite{Co}).
Two permutation representations on the orbits $\Omega_6$  and $\Omega_6'$ of length 6 differ by the nontrivial outer automorphism of  $\mathfrak S_6$. 
The union $\Omega_6 \cup \Omega_6'$ is an umbral dodecad, 
and so is its complement $\Omega_2 \cup \Omega_{10}$.
Thus $\mathfrak S_6$ is embedded into $M_{12}$ in two different ways.

Among the eleven groups  ($\ast$), three groups $\mathfrak S_5, N_{72}, \mathfrak A_6$  are subgroups of  $\mathfrak S_6$.
(The second one, $N_{72}\simeq C_3^2\splitext D_8$, is the normalizer of a $3$-Sylow subgroup.)
By the embedding  $\mathfrak S_6 \hookrightarrow M_{24}$ above mentioned, they are also subgroups of $M_{11}$.
More precisely we have the following:

\begin{lem}\label{embM11}
%$(1)$ 
Each of three groups $\mathfrak S_5, N_{72}, \mathfrak A_6$ is embedded into  $M_{11}$  so that it decomposes  the operator domain $\Omega_+\setminus \{\star\}$ into two orbits. 
The orbit length $\{a, b\}$  are $\{5,6\}$, $\{2,9\}$, $\{1,10\}$, respectively.
\end{lem}
In fact, the three groups are isomorphic to the stabilizer of  $\mathfrak S_6 \act \Omega_i$ with  $i=6, 10, 2$, respectively.

%Let $a$ and $b$  be their length.
%The set $\{a,b\}$ is $\{5,6\}$, $\{1,10\}$ and $\{2,9\}$ respectively.

As is well-known, the free part $H^2(S, \Z)_f \simeq \Z^{10}$  of the second cohomology group of an Enriques surface  $S$, equipped with the cup product,  is isomorphic to 
the lattice associated with the diagram  $T_{2,3,7}$ (Figure \ref{T237}).
\begin{figure}
\centering
\begin{picture}(160,55)
\linethickness{0.7pt}
\put(40,0){\circle*{6}}
\multiput(0,20)(20,0){9}{\circle*{6}}
\put(0,20){\line(1,0){160}}
\put(40,0){\line(0,1){20}}
\put(-3,25){$r_1$}
\put(17,25){$r_2$}
\put(37,25){$r_3$}
\put(57,25){$r_4$}
\put(77,25){$r_5$}
\put(97,25){$r_6$}
\put(117,25){$r_7$}
\put(137,25){$r_8$}
\put(157,25){$r_9$}
\put(44,-4){$r_0\leftrightarrow\mathfrak A_6$}
\put(17,36){$\updownarrow$}
\put(15,48){$N_{72}$}
\put(57,36){$\updownarrow$}
\put(55,48){$\mathfrak S_5$}
\end{picture}
\caption{The diagram $T_{2,3,7}$}
\label{T237}
\end{figure}
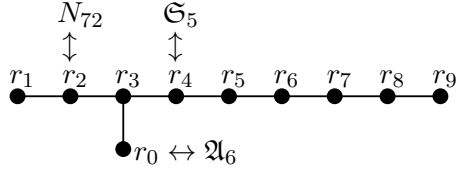
The Weyl group of $T_{2,3,7}$ contains $\mathfrak S_5 \times \mathfrak S_6, \mathfrak S_2 \times \mathfrak S_9$ and $\mathfrak S_{10}$ as Weyl subgroups. 
(The three subgroups become visible by removing the 
corresponding vertices shown in Figure \ref{T237}.)
Hence, via $M_{11}$ and the Weyl groups, 
each of the three groups $G = \mathfrak S_5, N_{72}, \mathfrak A_6$ acts isometrically on $H^2(S, \Z)_f$.
The invariant part $H^2(S, \Z)_f^G$ is generated by an element of square length $ab$ and its orthogonal complement is isomorphic to the root lattice $A_{a-1} \oplus A_{b-1}$.
%\footnotetext{Note that the Weyl group of $T_{p,q,r}$ contains $\mathfrak S_{p-1} \times \mathfrak S_{q+r}$ as a Weyl subgroup. The three actions become visible by removing the 
%corresponding vertices shown in Figure \ref{T237}.}
In this and next sections, 
%what follows, %for each of the three groups  $G = \mathfrak S_5, N_{72}, \mathfrak A_6$,  
we construct a semi-symplectic action of  $G$ on an Enriques surface $S$ which is not only Mathieu but also realizes the above $G$-action on $H^2(S,\Z)_f$.
Note that, since these groups are generated by involutions, the action is 
automatically semi-symplectic by Proposition \ref{2356}.
%related more closely with its embedding into $M_{11}$.

\begin{thm}\label{veryMathieu}
The cohomological action $G \act H^2(S, \Z)_f$ of three groups in Examples~\ref{Ohashi}, \ref{Mukai} and \ref{KOZex} are $G$-equivariantly isomorphic to the one described above.
In particular, the actions $G \act S$ in Examples~\ref{Ohashi}, \ref{Mukai} and 
\ref{KOZex} are Mathieu, proving  $(3) \Rightarrow (1)$ of Theorems~\ref{main} and \ref{main2}.
Furthermore, the $G$-invariant (primitive) polarization is unique (up to numerical equivalence) 
and is of degree $ab$.
\end{thm}

We begin with some lattice-theoretic lemmas and then proceed to construct the 
group actions, proving the theorem for each group.

\subsection{Some lattice theory}
Let $h$ be a nef and big divisor on an Enriques surface $S$.
By \cite[Lemma 2.9 and (2.11)]{Cossec85}, the quantity
\[\Phi (h)=\min \{ (f,\!h)\mid \text{$f\in H^2(S,\Z)_f$ is primitive 
with $(f^2)=0$ and $(f,\!h)>0$.}\}\]
is attained by a nef isotropic element $f$.
We call $\Phi$ the {\em{gonality function}} and
a {\em{gonality half-pencil}} $f$ for $h$ is an isotropic element $f\in H^2(S,\Z)_f$
satisfying $(f,h)=\Phi(h)$.
%Moreover, when $h$ is ample, {\em{every}} gonality half-pencil $f$ is nef. 
We see that if $h$ is a $G$-invariant polarization, 
$G$ permutes gonality half-pencils for $h$.
\begin{lem} We have the following.
\begin{enumerate}\label{gonalities}
\item Polarizations $h$ with $(h^2)=10$ and $\Phi(h)=3$ are unique up to 
the orthogonal group ${\rm O}(T_{2,3,7})$. Moreover, there are exactly ten
gonality half-pencils for such $h$.
\item The same holds for $(h^2)=18$ and $\Phi(h)=4$. Moreover, there are 
exactly nine gonality half-pencils for such $h$.
\item The same holds for $(h^2)=30$ and $\Phi (h)=5$. Moreover, there are 
exactly six gonality half-pencils for such $h$.
\end{enumerate}
\end{lem}
\begin{proof}
In each item, the former assertion is an easy consequence of the construction of the dual basis 
$b_i\ (0\leq i\leq 9)$ of $r_i$ (in Figure \ref{T237}) 
as in \cite[(1.3)]{Cossec85}. We have $h=b_0$,
$h=b_2$ and $h=b_4$ respectively. For the latter assertions, we note the following decompositions
of $h$ into isotropic elements in terms of the isotropic sequence $f_1,\dots,f_{10}$:
\begin{equation*}
%\begin{split}
b_0=\frac{f_1+\cdots+f_{10}}{3},\ 
b_2=\frac{(b_0-f_1-f_2)+f_3+\cdots+f_{10}}{2},\ 
b_4=f_5+\cdots+f_{10}.
%\end{split}
\end{equation*}
By Cauchy-Schwarz inequality, we see that the sets $\{f_i\mid 1\leq i\leq 10\}$,
$\{b_0-f_1-f_2,f_3,\dots,f_{10}\}$ and $\{f_i\mid 5\leq i\leq 10\}$ give gonality half-pencils respectively. We note that the number of gonality half-pencils equals
$\max \{a,b\}$ via Figure \ref{T237} where $a,b$ are as in Lemma \ref{embM11}.
%The intersection matrix $(b_i,b_j)$ is available in \cite{Dolgachev-survey}.
\end{proof}

\subsection{Mathieu action of $\mathfrak S_5$}\label{subsection:S5}
We consider the surface in $\mathbb{P}^4$ defined by 
\begin{equation}\label{A9+A1}
\overline{X} \colon \sum_{1\leq i<j \leq 5}x_ix_j=\sum_{1\leq i<j \leq 5} \frac{1}{x_ix_j}=0,
\end{equation}
which has five nodes at the coordinate points and whose minimal desingularization $X$ 
is a $K3$ surface. The Cremona transformation $\varepsilon \colon (x_i)\mapsto (1/x_i)$ 
induces a free involution and we let $S$ be the Enriques surface $X/\varepsilon$.
The symmetric group $\mathfrak S_5$ acts on $X$ and $S$ by permutations of coordinates.\\

\noindent {\it Proof of Theorem \ref{veryMathieu} for $G=\mathfrak{S}_5$}.
The key is to construct a good (rational) generator set of the second cohomology.
The exceptional curves at nodes are interchanged with the rational curves $\overline{X}\cap \{x_i=0\}$ by $\varepsilon$. Thus we get five smooth rational curves $r_1,\dots,r_5$ 
on $S$ whose dual graph is the complete graph with doubled edges and five vertices, denoted $K_5^{[2]}$. Also, for each even involution $\sigma=(ij)(kl)\in \mathfrak{A}_5$, 
we can find lines
\begin{equation*}
\begin{split}
l'_{\sigma} \colon\ \ x_i\colon x_j\colon x_k\colon x_l=1\colon -1\colon \sqrt{-1}\colon -\sqrt{-1}, \\
l''_{\sigma} \colon\ \ x_i\colon x_j\colon x_k\colon x_l=1\colon -1\colon -\sqrt{-1}\colon \sqrt{-1},
\end{split}
\end{equation*}
lying in $\overline{X}$ and interchanged by $\varepsilon$. 
Hence we have another
$15$ smooth rational curves $l_{\sigma}$ on $S$. 
The incidence relation is given by
\begin{equation*}
(l_{\sigma},l_{\tau})=
\begin{cases}
-2 & \text{ if } \sigma=\tau,\\
1 & \text{ if } \sigma\tau\text{ has order $3$},\\
0 & \text{ otherwise,}\\
\end{cases}
\end{equation*}
and their dual graph is isomorphic to the line graph\footnotemark
$L(P)$ of the famous Petersen graph  $P$ with $10$ vertices and $15$ edges.
\footnotetext{The line graph, or edge graph, $L(\Gamma)$ of a graph $\Gamma$ is the one whose vertices correspond to edges of $\Gamma$ and two vertices 
are connected by an edge if they share a vertex in $\Gamma$.}
\vspace{1cm}
\begin{figure}
\centering
\begin{picture}(126,118)(-63,-55)
\linethickness{0.7pt}
\put(0,55){\circle*{6}}
\put(0,30){\circle*{6}}
\multiput(-52.3,17)(104.6,0){2}{\circle*{6}}
\multiput(-28.5,9.3)(57,0){2}{\circle*{6}}
\multiput(-17.6,-24.3)(35.2,0){2}{\circle*{6}}
\multiput(-32.3,-44.5)(64.6,0){2}{\circle*{6}}
\put(0,55){\line(52.3,-38){52.3}}
\put(0,55){\line(-52.3,-38){52.3}}
\put(52.3,17){\line(-20,-61.5){20}}
\put(-52.3,17){\line(20,-61.5){20}}
\put(-32.3,-44.5){\line(1,0){64.6}}
%%%%%%%%%%%%%%
\put(0,55){\line(0,-1){25}}
\put(52.3,17){\line(-23.8,-7.7){23.8}}
\put(-52.3,17){\line(23.8,-7.7){23.8}}
\put(-32.3,-44.5){\line(14.7,20.2){14.7}}
\put(32.3,-44.5){\line(-14.7,20.2){14.7}}
%%%%%%%%%%%%%
\put(-28.5,9.3){\line(1,0){57}}
\put(0,30){\line(17.6,-54.3){17.6}}
\put(0,30){\line(-17.6,-54.3){17.6}}
\put(28.5,9.3){\line(-46.1,-33.6){46.1}}
\put(-28.5,9.3){\line(46.1,-33.6){46.1}}
%\put(40,0){\line(0,1){20}}
\end{picture}
\caption{The Petersen graph}
\label{Petersen}
\end{figure}
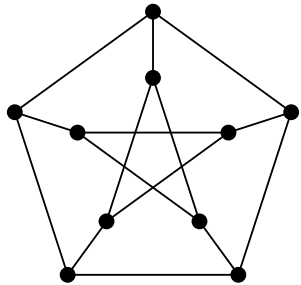
%\begin{figure}[h]
%\includegraphics[width=2.5cm,bb=0 0 77 77]{MIME1-images.jpeg}  % for local compiling  
%%\includegraphics{MIME1-images.jpeg}  %for math arxiv
%\caption{Petersen graph}
%\end{figure}
The connections between $K_5^{[2]}$ and $L(P)$ are given by 
\begin{equation*}
(r_i,l_{\sigma})=\begin{cases}
2 & \text{if $\sigma (i)=i$}\\
0 & \text{otherwise}.
\end{cases}
\end{equation*}
The $20$ curves in $K_5^{[2]}\cup L(P)$ generate $H^2(S,\Z)_f$ up to index two.
The overlattice structure is given, for example, by adding the half-pencil of the 
elliptic fibration defined by a pair of disjoint pentagons in $L(P)$. 
There are six such pairs in $L(P)$, hence we have six elliptic 
pencils on $S$ with reducible fibers of type $I_5+I_5$. We denote them by 
$|2f_j|\ (1\leq j\leq 6)$. These classes satisfy $(f_i, f_j)=1-\delta_{ij}$. 

Now, since the symplectic group action by $\mathfrak{S}_5$ is maximal
(for $K3$ surfaces), we have the relation 
\[\sum_{i=1}^5r_i = \sum_{j=1}^6 f_j  \in H^2(S,\Z)_f\]
and this gives the $\mathfrak{S}_5$-invariant polarization $h$ of degree $30$.
Those $f_j$ are exactly the gonality half-pencils for $h$ as in the proof of Lemma \ref{gonalities}.
Moreover, the orthogonal complement of $h$ is spanned by the nine elements
\[(r_i-r_{i+1})/2\ (1\leq i\leq 4),\quad f_j-f_{j+1}\ (1\leq j\leq 5), \]
which is isomorphic to $A_4+A_5$.
The action of $\mathfrak{S}_5$ is isomorphic to that on the root lattice,
therefore we obtain Theorem \ref{veryMathieu} for Example \ref{Ohashi}.
(Every $r_i-r_{i+1}$ is divisible by 2 since $r_1+r_2$  is equivalent to $2(l_{(12)(34)}+l_{(12)(35)}+l_{(12)(45)})$
and so on.)

\begin{rmk}\label{kondo7}
We note that our Enriques surface $S$ is isomorphic to the Enriques surface of type VII
in \cite{kondo86}. In fact, the configuration of smooth rational curves we studied on $S$ is 
the same as that of type VII. In particular, $S$ has a finite automorphism group and 
by the main theorem of \cite{kondo86} we obtain the assertion. 
Our equation \eqref{A9+A1} will be valuable by its simplicity and ease to work with.
\end{rmk}
\begin{rmk}
We can eliminate the variable $x_5$ from \eqref{A9+A1} and we get then the 
symmetric quartic surface 
$s_2^2=s_1s_3$,
where $s_i$ are the fundamental 
symmetric polynomials in $x_1,\dots,x_4$. It has nodes at the coordinate points
and has an obvious $\mathfrak{S}_4$ action, hence as in \cite{Robfest} we have 
a homomorphism 
\[\varphi \colon \mathfrak{S}_4\ltimes (C_2^{*4})\rightarrow \mathrm{Aut}(S),\]
where the generators of $C_2$'s are the covering involutions of the projection 
from coordinate points.
Here we find that, opposed to the cases we treated in \cite{Robfest},
$\varphi$ is not an isomorphism.
In fact, the four involutions corresponding to projections are nothing but 
the transpositions $(x_i x_5)$ and the image of $\varphi$ is
the finite group $\mathfrak{S}_5$. In particular, it is not injective.
By Remark \ref{kondo7}, we see that $\mathrm{Aut}(S)\simeq \mathfrak{S}_5$ 
and in fact $\varphi$ is surjective.
\end{rmk}
\begin{rmk}\label{Clebsch}
Similar to \eqref{A9+A1}, an Enriques surface with a (semi-symplectic) action of $\mathfrak S_5$ is obtained also from the quartic surface
\begin{equation}\label{VI}
\sum_{i=1}^5 x_i=\sum_{i=1}^5 \frac{1}{x_i}=0.
\end{equation}
The subaction of the alternating group $\mathfrak A_5$  is Mathieu
but the full action is not Mathieu.
(The quartic surface is the Hessian of the Clebsch diagonal cubic surface 
$\sum_{i=1}^5 x_i= \sum_{i=1}^5 x_i^3=0,$
and this Enriques surface is Kondo's  of type VI in \cite{kondo86}.
See \cite[Remark 2.4]{D-vG}.)
\end{rmk}

\section{Enriques surfaces of Hesse-Godeaux type}\label{section:Hesse-Godeaux}
In this section we prove Theorem~\ref{veryMathieu}  for  $N_{72}$  and $\mathfrak A_6$.
For that purpose we consider the $K3$ surface in $\mathbb{P}^5$
\begin{equation}\label{Hesse-Godeaux}
X=X_{\lambda,\mu}\colon 
\begin{cases} x_0^2 - \lambda x_1x_2 = y_0^2 - \mu y_1y_2\\
x_1^2 - \lambda x_0x_2 = y_1^2 - \mu y_0y_2\\
x_2^2 - \lambda x_0x_1 = y_2^2 - \mu y_0y_1
\end{cases} 
(\lambda \ne \mu\text{ and }\lambda, \mu \ne 1, \omega, \omega^2), 
\end{equation}
where $\omega=(-1+\sqrt{-3})/2$,
and its Enriques quotient $S$ by the free involution 
$\varepsilon\colon (x:y)\mapsto (x:-y)$. 
The abelian group $C_3^2$ acts on both $X$ and $S$ explicitly by 
\begin{equation}\label{3^2}
\begin{split}
\alpha \colon 
(x_0:x_1:x_2:y_0:y_1:y_2)& \mapsto (x_1:x_2:x_0:y_1:y_2:y_0), \\ 
\beta \colon 
(x_0:x_1:x_2:y_0:y_1:y_2)& \mapsto (x_0:\omega x_1:\omega^2 x_2:y_0:\omega y_1:\omega^2 y_2).
\end{split}
\end{equation}
These surfaces have closer relation with the 
rational elliptic surface $R$ given by the Hesse pencil of plane cubics
\begin{equation}\label{Hesse cubic}
z_0^3 + z_1^3 + z_2^3 -3\kappa z_0z_1z_2=0
\end{equation}
and its two fibers at $\kappa =\lambda,\mu$. We use the following lemma.
\begin{lem}\label{isom}
Let $g \colon \P^2\rightarrow \P^2$ be a collineation (namely a projective linear automorphism) 
preserving the Hesse pencil \eqref{Hesse cubic}. It induces a M\"{o}bius transformation on the parameter 
$\kappa$ of \eqref{Hesse cubic} and we denote it by the same letter $g$. 
Then, for any such $g$, we have an induced isomorphism between the 
Enriques surfaces of Hesse-Godeaux type
\begin{equation}
\begin{split}
\tilde{g} \colon X_{\lambda, \mu}&\rightarrow X_{g(\lambda ), g(\mu )}\\
 (x:y)&\mapsto (c(\lambda,\mu)g(x):g(y))\\
\end{split}
\end{equation}
for a suitable constant $c(\lambda,\mu)$ depending on $\lambda,\mu$ and $g$. 
\end{lem}
\begin{proof}
We denote $g$ by the matrix form $(g_{ij})\in \mathrm{GL}(3,\C)$ and use the notation 
$z'=g(z)\in \P^2$, $\kappa'=g(\kappa)\in \P^1$. 
We have the relation
\[z_0'^3+z_1'^3+z_2'^3-3\kappa'z_0'z_1'z_2'=c(\kappa)(z_0^3+z_1^3+z_2^3-3\kappa z_0z_1z_2)\]
where $c(\kappa)$ is a constant depending only on $k$ and $g$. 
Differentiating both sides, we get
\[\sum_{i=0}^2 (z_i'^2-\kappa'z_{i-1}'z_{i+1}')g_{ij}=c(k)(z_j^2-\kappa z_{j-1}z_{j+1})\ (j=0,1,2).\]
This equation applied to $\kappa =\lambda, \mu$ shows that $(x:y)\in X_{\lambda,\mu}$
if and only if $(g(x)/\sqrt{c(\lambda)}:g(y)/\sqrt{c(\mu)})\in X_{g(\lambda),g(\mu)}$.
\end{proof}

The group of collineations preserving the pencil \eqref{Hesse cubic} is the semi-direct product $C_3^2\splitext SL(2, \F_3)$ of the above $C_3^2$ by the binary octahedral group.
It is of order $216$ and called the {\em{Hessian group}}. 
We refer the readers to \cite{Blichfeldt} for the explicit generators of this group. 

In what follows we exhibit some elliptic pencils on $X$.
First, let $C_k\ (k=0,1,2)$ be the conic on $X$ defined by $x_i=y_i\ (i\neq k), \lambda x_k=\mu y_k$. It is easy to see that the divisor 
$C:=C_0+C_1+C_2$ constitutes a $\P^1$-configuration of Kodaira type $I_3$. 
For an element $g$ in the Hessian group, let $C_g$ be the pullback 
$\tilde{g}^* (C)$ of the corresponding configuration $C\subset X_{g(\lambda),g(\mu)}$ via 
Lemma \ref{isom}, where $g$ runs over the following
\[g=\begin{pmatrix}
1 & 1 & 1 \\
\omega & \omega^2 & 1\\
\omega^2 & \omega & 1\\
\end{pmatrix},\ 
\begin{pmatrix}
1 & 1 & \omega \\
\omega & 1 & 1\\
1 & \omega & 1\\
\end{pmatrix},\ 
\begin{pmatrix}
1 & 1 & \omega^2 \\
\omega^2 & 1 & 1\\
1 & \omega^2 & 1\\
\end{pmatrix},
\text{ and }\mathrm{id_R}.
\]
It turns out that these four $I_3$ configurations are disjoint each other and define
an elliptic pencil $F_{\infty}$ on $X$. It is invariant under $\varepsilon$, and we denote 
the induced pencil on $S$ by $|2f_{\infty}|$. It has therefore $4$ fibers of type $I_3$ and 
we call the $12$ rational curves the {\em{h-conics}}.
The pencil is called the {\em{primary pencil}}. 
Note that $h$-conics are stable under the action \eqref{3^2}
of the group $C_3^2$ on $X$ and $S$.
These actions can be seen as induced from the Mordell-Weil group of $R$.

For the second pencils, we make use of 
%let us denote by 
the difference 
\[(x_0-x_1)(x_0+x_1+\lambda x_2)=(y_0-y_1)(y_0+y_1+\mu y_2),\]
of the 1st and 2nd defining equations \eqref{Hesse-Godeaux} of $X$. %reads as
By this equality, the rational functions 
$F_{00}=(y_0-y_1)/(x_0-x_1)$ and $F'_{00}=(y_0-y_1)/(x_0+x_1+\lambda x_2)$
are elliptic parameters on $X$ and their squares $F_{00}^2, (F'_{00})^2$ define 
elliptic pencils $|2f_{00}|, |2f'_{00}|$ on $S$. 
By applying the $C_3^2$-action \eqref{3^2}, 
we obtain $18$ elliptic parameters on $X$ as follows.
\[F_{kl} = \frac{y_k-\omega^l y_{k+1}}{x_k-\omega^l x_{k+1}},\ 
F'_{kl} = \frac{y_k-\omega^l y_{k+1}}{x_k+\omega^l x_{k+1}+\lambda \omega^{2l}x_{k+2}}
(k,l\in \{0,1,2\}=\mathbb{Z}/3).\]
We denote by $|2f_{kl}|,\ |2f'_{kl}|$ the induced elliptic pencils on $S$. 
%They are called the {\em{secondary pencils}}.
\begin{lem}
We have the following relations between numerical classes.
$(1)$ $h=f_{kl}+f'_{kl},\ (h,f_{kl})=2,\ (f_{kl},f_{k'l'})=1-\delta_{kk'}\delta_{ll'}$,
where  $\delta_{kk'}$ is the Kronecker delta and $h\in H^2(S,\Z)_f$ is the natural polarization of degree $4$.
In particular, 
$\{h,f_{00},\dots,f_{22}\}$ is a basis of the $\Q$-vector space
$H^2(S,\Q)$. \\
$(2)$ $(h,f_{\infty})=3$ and $(f_{\infty},f_{kl})=1\ (\forall k,l)$.
\end{lem}
\begin{proof}
(1) and the first equality in (2) are deduced by definition. 
For the last equalities, it suffices to compute $(f_{\infty},f_{00})$
since $f_{\infty}$ is invariant under the subgroup $C_3^2$.
%Recall that the pullback $F_{00}$ of $f_{00}$ is defined by the elliptic parameter 
%$F_{00}= (y_0-y_1)/(x_0-x_1)$. 
We use the representative $F_{\infty}=C_0+C_1+C_2$.
By definition of $F_{00}$, we see that $(F_{00}, C_2)=0$ and 
$(F_{00}, C_i)=1\ (i=0,1)$. 
Therefore, $(f_{\infty},f_{00})=(F_{\infty},F_{00})/2= \sum_{i=0}^2(C_i,F_{00})/2 =1$. 
\end{proof}
%\begin{lem}
%\begin{enumerate}
%\item Let $C_k\ (k=0,1,2)$ be the conic in $X$ defined by $x_i=y_i\ (i\neq k), \lambda x_k=\mu y_k$. Then $C_0+C_1+C_2$ is a $\P^1$-configuration of Kodaira type $I_3$ which is disjoint from $F_{\infty}$. 
%In particular, $F_{\infty}$ and $C_0+C_1+C_2$ are linearly equivalent on $X$. 
%\item We have $(h,f_{\infty})=3$, $(f_{\infty},f_{kl})=1\ (\forall k,l)$.
%\end{enumerate}
%\end{lem}

The $12$ $h$-conics and $9$ elliptic pencils 
$|2f_{00}|,\dots,|2f_{22}|$ have the following outstanding property.
% form a Steiner system St(2,3,9).
\begin{lem}\label{first Steiner}
$(1)$ Every $h$-conic is contained as fiber in exactly three of the 9 elliptic fibrations defined by $|2f_{00}|,\dots,|2f_{22}|$.\\
$(2)$ For every pair $|2f_{ij}|, |2f_{kl}|$ of elliptic pencils, there exists a unique $h$-conic which is contained in the fibers of elliptic pencils defined by them.
\end{lem}
The proof is a simple computation and we omit it. By this lemma, we obtain a Steiner 
system St(2,3,9) on the set of 9 elliptic pencils, where the special triplet is the set of 
pencils containing a (fixed) $h$-conic.

An easy counting shows the following.
\begin{cor}\label{9_4-12_3}
Every elliptic fibration defined by $|2f_{ij}|$ contains four $h$-conics in different fibers.
\end{cor}
For example, the elliptic parameter  $(y_0-y_1)^2/(x_0-x_1)^2$ of  $|2f_{00}|$ has four critical values
$$1, \frac{\lambda -1}{\mu - 1}, \frac{\lambda - \omega}{\mu - \omega}, \frac{\lambda - \bar\omega}{\mu - \bar\omega},$$
at which the fiber contain an $h$-conic.
Therefore, the elliptic pencils and the $h$-conics form the (Hesse) $(9_4,12_3)$-configuration.

\subsection{Mathieu action of $N_{72}$.}\label{subsection:N72}
In this subsection we define a non-linear automorphism of Enriques surfaces of Hesse-Godeaux type to construct an $N_{72}$-action.
%The extension of this action of $C_3^2\splitext C_4$ to $N_{72}$ is difficult, not linear and done as follows. 
For that purpose, we assume that 
\begin{equation}\label{special relation}
(\lambda +1)(\mu  + 1) = 1
\end{equation}
holds and consider the matrix $A=(a_{kl})_{0\leq k,l\leq 2}$
defined by 
\begin{equation}\label{5in8}
 A=\begin{pmatrix}
\mu x_0 & x_2+cy_2 & x_1-cy_1\\
x_2-cy_2 & \mu x_1 & x_0+cy_0\\
x_1+cy_1 & x_0-cy_0 & \mu x_2
\end{pmatrix}, 
\end{equation}
where $c$ is a constant satisfying $c^2=1-\mu^2$.
We regard $(a_{kl})$ as the homogeneous coordinates of $\mathbb{P}^8$.
Then the above expression \eqref{5in8} defines an embedding $\mathbb{P}^5 \subset \mathbb{P}^8$ whose linear equations are given by
\[2a_{kk}=\mu(a_{lm}+a_{ml}),\quad \{k,l,m\}=\{0,1,2\}.\]
Let $\delta$ be the correspondence 
$A\mapsto A^{\mathrm{adj}}=(\widetilde{a}_{kl})_{0\leq k,l\leq 2}$, 
the adjoint matrix (namely, the transposed cofactor matrix) of $A$.\footnote{ 
For example, the $(0,1)$-entry $\widetilde{a}_{01}$ is given by 
$-\mu x_2(x_2+cy_2)+(x_0-cy_0)(x_1-cy_1)$.}
By $(A^{\mathrm{adj}})^{\mathrm{adj}}=(\det A)A$,
$\delta$ defines a birational involution of $\mathbb{P}^8$.
Moreover, we can check that the defining equations \eqref{Hesse-Godeaux} of the $K3$ surface $X$ is
nothing but the equations 
\[2\widetilde{a}_{kk}=\mu(\widetilde{a}_{lm}+\widetilde{a}_{ml}),\quad \{k,l,m\}=\{0,1,2\}\]
for the adjoint matrix  $A^{\mathrm{adj}}$.
This shows that $\delta$ restricts to an involution on $X$.
In terms of new coordinates, $\varepsilon$ on $\mathbb{P}^5$ 
is identified with the involution $A\mapsto \!^tA$.
Since the coordinate sets $\{x_i\}$ and $\{y_i\}$ form the basis of the 
positive and negative eigenspaces of this involution respectively, 
we see that $\delta$ commutes with $\varepsilon$
and acts on $S$, too.

Now we consider the most special Enriques surface of 
Hesse-Godeaux type with parameter $\lambda, \mu=1 \pm \sqrt{3}$, which satisfies \eqref{special relation}.
The fibers of \eqref{Hesse cubic} at $k=\lambda,\mu$ are
preserved by the collineation 
\[g=\begin{pmatrix}
1&1&1 \\ 1&\omega&\omega^2 \\ 1&\omega^2&\omega
\end{pmatrix}.\]
Via Lemma \ref{isom}, it induces an automorphism of 
$X_{\lambda,\mu}$. Let $\gamma$ be the one induced on $S$.
%explicitly, given by 
%\begin{equation}\label{4}
%\gamma\colon \begin{pmatrix}
%x_0 & y_0\\
%x_1 & y_1\\
%x_2 & y_2
%\end{pmatrix}
%\mapsto
%\begin{pmatrix}
%x_0+x_1+x_2 & -\sqrt{-1}(y_0+y_1+y_2)\\
%x_0+\omega x_1+\omega^2 x_2 & -\sqrt{-1}(y_0+\omega y_1+\omega^2 y_2)\\
%x_0+\omega^2 x_1+\omega x_2 & -\sqrt{-1}(y_0+\omega^2 y_1+\omega y_2)
%\end{pmatrix}.
%\end{equation}
The group of automorphisms on $S$ generated by $\alpha, \beta, \gamma, \delta$ is denoted by $G$.
It is easy to see that the linear automorphisms $\alpha, \beta, \gamma$ generate a subgroup 
$C_3^2\splitext C_4$ of order $36$.
To see the relations among $\alpha, \beta,\gamma$ 
and the non-linear $\delta$, we note that the linear ones
$\alpha, \beta$ and $\gamma$ can be identified with the 
automorphisms $A\mapsto \!^tB AB$, where matrices $B$ are given by 
\[B_{\alpha}=\begin{pmatrix} 0 & 0 & 1 \\
1 & 0 & 0 \\ 0 & 1 & 0\end{pmatrix},
B_{\beta}=\begin{pmatrix} 1 & 0 & 0 \\
0 & \omega^2 & 0 \\ 0 & 0 & \omega\end{pmatrix},
B_{\gamma}=\begin{pmatrix} 1 & 1 & 1 \\
1 & \omega^2 & \omega \\ 1 & \omega & \omega^2\end{pmatrix}.
\] 
From this observation we easily obtain the equalities 
\[\delta \alpha=\alpha\delta,\ \delta \beta=\beta^2
\delta,\ \delta \gamma=\gamma^{-1}\delta.\] 
Thus, the group $G$ is really the one $C_3^2\splitext D_8\simeq 
N_{72}$.\\

\noindent {\it Proof of Theorem \ref{veryMathieu} for $G=N_{72}$}.
%Next, 
Let us study the action of the Cremona involution $\delta$ on cohomology.
Let $I$ be the indeterminacy locus of $\delta$ as a transformation of 
$\P^5$. It is obtained by cutting the Segre variety 
\[\Sigma_{2,2}=\{A\in \P^8\mid \mathrm{rank}\, A\leq 1\}\]
three times by hyperplanes, hence is an elliptic curve of degree $6$. 
(The smoothness of the linear section is easily seen
by using two projections $I\rightarrow \P^2$.)
In fact, since $I$ and the curve $C_0+C_1+C_2$ are disjoint, $I$ is linearly equivalent to the 
primary pencil $F_{\infty}$.

Since $\delta$ is defined by the linear system of quadrics containing $I$, 
we obtain $\delta^* (h)=2h-f_{\infty}$. It follows that $\tilde{h}=h+\delta^*(h)$ 
is the polarization of degree $18$ invariant under $G=N_{72}$.
By looking at the intersection numbers with generators of $H^2(S,\Q)$, we get 
$2f_{\infty}=-3h+\sum_{k,l}f_{kl}$ and we see 
\begin{equation*}
%\begin{split}
\tilde{h}=3h-f_{\infty}=\frac12\sum_{k,l} (h-f_{kl})
=\frac12(f'_{00}+\cdots+f'_{22}).
%=\frac{f'_{00}+\cdots+f'_{22}}{2}.
%\end{split}
\end{equation*}
Therefore, $f'_{kl}\ (0\leq k,l\leq 2)$ 
are the gonality half-pencils for $\tilde{h}$. The orthogonal complement of $\tilde{h}$ is 
spanned by $h-f_{\infty}$ and the differences $f'_{kl}-f'_{k'l'}\ (k,l,k',l'\in \{0,1,2\})$.
%;\ f'_{00}-f'_{01}, f'_{01}-f'_{02}, f'_{02}-f'_{10},\dots, f'_{21}-f'_{22}\]
They generate the lattice $A_1+A_9$ and we obtain Theorem \ref{veryMathieu}
for Example \ref{Mukai}.

\subsection{Mathieu action of $\mathfrak{A}_6$.}\label{subsection:A6}
We continue to assume that $\lambda, \mu=1\pm \sqrt{3}$.
As we saw the primary elliptic fibration  $|2f_\infty|$  is of {\it Hesse type}, that is, has four singular fibers of type $I_3$, %$\tilde A_2$,
and each of the nine elliptic fibrations $|2f_{00}|,\dots, |2f_{22}|$ has four reducible fibers (Corollary~\ref{9_4-12_3}). 
On our special surface, each reducible fiber of  $X \to \P^1$  becomes the union of a conic and two lines.
An example of such line in $X$  is given in parametrization by
%\begin{equation}
\begin{multline}
(x_0:x_1:x_2:y_0:y_1:y_2) \\
= ((1-\sqrt{3})t: t + \sqrt{3}: t- \sqrt{3}: 1+ \sqrt{3}: \sqrt{3}t+1: -\sqrt{3}t+1).
\end{multline}
%\end{equation}
In fact, the elliptic parameters of $|2f_{01}|, |2f_{02}|, |2f_{21}|, |2f_{22}|$ takes constant values on this line.
In total,  36 lines appear on  $X$ in such a way and we obtain 18 smooth rational curves on  the Enriques surface  $S$. We call them {\it $h$-lines}.
%By applying the $C_3^2$ action, we obtain 36 lines on  $X$,  and 18 smooth rational curves on  $S$.
%Each of the nine elliptic fibrations   $|2f_{00}|,\dots, |2f_{22}|$ is also of Hecke type.
These rational curves and 9 elliptic pencils  $|2f_{00}|,\dots,|2f_{22}|$  
satisfy the following:

\begin{lem}\label{second Steiner}
$(1)$ Every $h$-line is contained as fiber in exactly four of the 9 elliptic fibrations defined by $|2f_{00}|,\dots,|2f_{22}|$.

$(2)$ For every triple $|2f_{ij}|, |2f_{kl}|,  |2f_{mn}|$ of elliptic pencils, there exists a unique $h$-line which is contained in the fiber of elliptic pencils defined by them.
\end{lem}
\begin{prop}
Let $T$ be the set of ten elliptic pencils $|2f_{ij}| (0\leq i,j\leq 2)$ and the primary 
pencil $|2f_{\infty}|$. 
We call a 4-subset of $T$ a {\it special quartet} if it is the set of pencils 
which contain a (fixed) $h$-conic or a (fixed) $h$-line. 
Then, the set of special quartets forms a Steiner system
$\mathrm{St}(3,4,10)$ on $T$. 
\end{prop}
\begin{proof} 
It follows from Lemmas \ref{first Steiner} and \ref{second Steiner}.
\end{proof}

\begin{rmk}\label{15A+15B}
The above 30 rational curves, say $C_\sigma$'s, are parametrized by the odd involutions $\sigma$ in the symmetric group  $\mathfrak S_6$.
There are two type of such involutions, 15 transpositions and 15 of permutation type $(2)^3$.
%$C_\sigma$ and $C_\tau$ are disjoint if and only 
The intersection number $(C_\sigma, C_\tau)$ is equal to 1
if $\sigma$ and $\tau$  do not commute.
When two distinct $\sigma$ and $\tau$ commute, $(C_\sigma, C_\tau)$ is equal to 0 if their permutation types are the same and to 2 otherwise.
\end{rmk}

Let  $f: S \to \P^1$  be an elliptic Enriques surface of Hesse type.
The Jacobian fibration is a rational elliptic surface induced from the Hesse pencil \eqref{Hesse cubic}.
Its Mordell-Weil group is isomorphic to $C_3^2$ and acts on $f: S \to \P^1$ by translation.
The following is the key for our construction of  $\mathfrak A_6$-action.

\begin{lem}
For an element  $H \in H^2(S, \Z)_f$, the following two conditions are equivalent:

$(1)$ $H$  is invariant under the induced action of $C_3^2$.

$(2)$ For every reducible fibre  $C_1+C_2+ C_3$ of  $f$, we have  
$$(H. C_1) = (H. C_2) = (H. C_3).$$
\end{lem}
\proof
Consider the two submodules  $A$  and  $B$ of  $H^2(S, \Z)_f$ consisting of  $H$ which satisfies (1) and (2), respectively, that is, the $C_3^2$-invariant part  $A$  and the equal degree part  $B$.
Both $H^2(S, \Z)_f/A$ and $H^2(S, \Z)_f/B$  are free modules.
Since the action of $C_3^2$  on $\{C_1, C_2, C_3\}$ is transitive on each of four reducible fibers, we have  $A \subset  B$.
By the computation of character (see Proposition \ref{chara}), $A$  is of rank 2.
Since the intersection  $B \cap F^\perp$  is generated by $F$, where $F$ is the fiber class,
$B$  is of rank at most 2.
Hence we have  $A=B$. 
\qed

Now we focus on the new polarization $H=h+f_\infty$, which is of degree 10.
All $h$-conics and $h$-lines are of degree 2 with respect to  $H$.
Hence, by the above lemma, the translations by Mordell-Weil groups of $|2f_{ij}|$ all 
preserve  $H$.
Since  $f_\infty, f_{00}, \ldots, f_{22}$ are the gonality half-pencils of  $H$,  
the translation acts on the set $\{f_\infty, f_{00}, \ldots, f_{22}\}$ by permutation
(Lemma \ref{gonalities} and its preceding sentence).
Let  $G$  be the automorphism group of  $(S, h)$  generated by all these translations.
Then $G$  contains 10 subgroups isomorphic to  $C_3^2$  and preserves the Steiner system St(3,4,10).
Hence we have a surjective homomorphism  $G \to \mathfrak A_6 \subset {\rm Aut}$ St(3,4,10)$\simeq \mathrm{Aut}(\mathfrak{A}_6)$.
Since $\mathfrak A_6$  is a maximal group which can act on a K3 surface symplectically, 
$G$  is isomorphic to $\mathfrak A_6$.
 
\medskip
\noindent {\it Proof of Theorem \ref{veryMathieu} for $G=\mathfrak{A}_6$}.

The polarization $H=h+ f_\infty$  on $S$ is of degree 10
and the invariant lattice $H^2(S,\Z)_f^{\mathfrak{A}_6}$ is spanned by it.
The orthogonal complement of $H$ in $H^2(S,\Z)_f$ is generated by the differences of 
the gonality half-pencils $f_{00},\dots,f_{22}$ and $f_\infty$.
Hence it is $\mathfrak{A}_6$-equivariantly isomorphic to the $A_9$-lattice and
we obtain Theorem \ref{veryMathieu} for Example \ref{KOZex}.

\begin{rmk}
The intersection of two semi-symplectic actions of $N_{72}$ and  $\mathfrak{A}_6$ on  $S$  is $C_3^2\splitext C_4$.
Hence we obtain a homomorphism $N_{72}*\mathfrak{A}_6 \to {\rm Aut}^{ss}\, S$ from the amalgam over $C_3^2\splitext C_4$.
In the forthcoming paper we will study the structure of ${\rm Aut}\, S$   using this homomorphism and Remark~\ref{15A+15B}.
\end{rmk}
%%%%%%%%%%%%%%%%%%%%%%%%%%%%%%%%%%%%%%%%%%%%%%%%%%%%%%%%%%%%%%%%%5
%\section{Proof of $(3) \Rightarrow (1)$ of the Main Theorem}\label{3->1}
\section{Mathieu actions of $C_2 \times \mathfrak A_4$ and $C_2\times C_4$}\label{3->1}
%~$\ref{main}$}

We prove that (3) of Theorem~\ref{main2} implies  (1) of Theorem~\ref{main}.
By Theorem~\ref{veryMathieu}, it suffices to prove for two groups $C_2 \times \mathfrak A_4$ and $C_2\times C_4$.
We realize their Mathieu actions on the Enriques surface  $S$  in Example~\ref{H192}.
The $K3$-cover of  $S$  is  a surface $X :  F(u,v,w) =0$  of tri-degree $(2,2,2)$ in  $\P^1 \times \P^1 \times \P^1$.
A nonzero 2-form on  $X$  is obtained as residue of the rational 3-form 
$du \wedge dv \wedge dw/F(u,v,w)$.
%on  $\P^1 \times \P^1 \times \P^1$,where $F(u,v,w)$  is the equation of  $X$.
Hence $X$  has the following symplectic actions.
%automorphisms.
\begin{anumerate}
\item %Changing even number of signs, that is, by 
The automorphism
$(u, v, w) \mapsto (\pm u, \pm v, \pm w)$
with even number of  \lq$-1$'s is symplectic.
Hence $C_2^2$  acts symplectically.
\item $(u, v, w) \mapsto (-\sqrt{-1}/u, -\sqrt{-1}/v, -\sqrt{-1}/w)$  is a symplectic involution.
\item For a permutation  $\sigma$  of $\{u, v, w\}$, 
$$(u, v, w) \mapsto (\sigma(u)^{\pm1}, \sigma(v)^{\pm1}, \sigma(w)^{\pm1})$$
is a symplectic automorphism if the parity of the number of  \lq$-1$'s  agree with that of  $\sigma$.
Hence  $\mathfrak S_4 = C_2^2\splitext\mathfrak S_3$  acts symplectically.
\end{anumerate}
(a) and (b)  generate a group isomorphic to $C_2^3$.
Hence the semi-direct product  $H= C_2^3\splitext \mathfrak S_4$  acts on  $X$ symplectically.
It is easily checked that this is $H_{192}$ (cf. the remark on p.\ 192 in \cite{mukai88}).
Since the involution  $\varepsilon: (u,v,w) \mapsto (-u,-v,-w)$  commutes with the above automorphisms,
$H$  acts semi-symplectically on the quotient Enriques surface  $S$.
The actions (a) are Mathieu since all involutions have only elliptic curves as fixed curves.
So is the action by (b) and the composite of (a) and (b) since they have only isolated fixed
points.
Hence the action (a)$\cdot$(b) of  $C_2^3$  and the automorphism $(u, v, w) \mapsto (v, w, u)$  generate a Mathieu action of  $C_2 \times \mathfrak A_4$.

The automorphism 
\begin{equation}\label{order 4}
(u, v, w) \mapsto (\sqrt{-1}u, \sqrt{-1}v, -\sqrt{-1}/w)
\end{equation}
is of order 4 and Mathieu.
In fact, it has exactly four fixed points, two of which are symplectic and the rest of which are anti-symplectic.
The automorphism  \eqref{order 4} and the involution (a) generate a Mathieu action of  $C_4 \times C_2$.
\qed

\begin{rmk}
The Enriques surface of  Example~\ref{H192} is the normalization of the sextic surface
$\sum_1^4 x_i^2+ \sqrt{-1}x_1x_2x_3x_4\sum_1^4 x_i^{-2}=0$
in $\P^3$.
The Mathieu actions of the two groups can be seen from this expression also.
See \cite[\S6]{Fields}.
\end{rmk}

\begin{rmk}\label{far from Mathieu}
The complete intersection of three diagonal quadrics 
%\begin{gather}
$$
\begin{cases}
x_1^2 + x_3^2 + x_5^2  = x_2^2 + x_4^2 + x_6^2  \\
x_1^2 + x_4^2 = x_2^2 + x_5^2 = x_3^2 + x_6^2 
\end{cases}$$
in  $\P^5$  is given in \cite[(0.4)]{mukai88} as a (smooth) $K3$ surface with a symplectic action of  $H_{192}$.
The automorphism  $(x_i) \mapsto ((-1)^ix_i)$  is a free involution and commutes with the action.
%This $K3$ surface also has a free involution which commutes with the action.
But the induced semi-symplectic action of $H_{192}=C_2^4:D_{12}$ on the Enriques quotient is far from Mathieu.
In fact, any (diagonal) involution in $C_2^4$ is not Mathieu.
Hence any sub-action of $C_2 \times \mathfrak A_4$ or $C_4 \times C_2$ is not Mathieu neither.
\end{rmk}
%%%%%%%%%%%%%%%%%%%%%%%%%%%%%%%%%%%%%%%%%%%%%%%%%%%%%%%%
\section{Semi-symplectic and Mathieu automorphisms}\label{preliminary}

Any Enriques surface is canonically doubly covered by a $K3$ surface. 
We always denote an Enriques surface by $S$ and the $K3$-cover by $X$. 
Let $\omega_X$ be a nowhere vanishing holomorphic $2$-form on $X$. An 
automorphism $\varphi$ is
{\em{symplectic}} if it preserves
the symplectic form $\omega_X$.
Equivalently, they are the elements in the kernel of the canonical representation $\Aut (X) \rightarrow \mathrm{GL} (H^0(\mathcal{O}_X (K_X)))$.
Along the same line of ideas, we define the following.

\begin{dfn}\label{semi-sympl}
Let $S$ be an Enriques surface.
An automorphism $\sigma\in \Aut (S)$ is {\em{semi-symplectic}} if it 
acts trivially on the space $H^0 (S, \mathcal{O}_S(2K_S))$.
\end{dfn}

The sections of $\mathcal{O}_S(2K_S)$ are identified with those of $\mathcal{O}_X(2K_X)$.
Recall that the covering involution $\varepsilon$ of $X/S$ negates $\omega_X$. 
Since for a given $\sigma\in \Aut(S)$ we have two lifts 
$\varphi_1$ and $\varphi_2=
\varphi_1 \varepsilon$ to automorphisms of $X$, we get the following proposition.

\begin{prop}\label{ssp}
$\sigma\in \Aut (S)$ is semi-symplectic if and only if one, say $\varphi_1$, of the two lifts 
is symplectic. Moreover, the other lift $\varphi_2=\varphi_1 \varepsilon$ negates $\omega_X$. 
\end{prop}
Since we can uniformly choose the symplectic lift, we have also
\begin{cor}\label{slift}
Let $G$ be a group of semi-symplectic automorphisms of $S$. Then the lifts of 
automorphisms in $G$ to the $K3$-cover $X$ constitute a group isomorphic to $G\times C_2$, where $C_2$ is generated by $\varepsilon$ and whose subgroup $G\times \{\id\}$
is the subgroup of symplectic automorphisms. 
\end{cor}

\begin{example}\label{gssp}
Let $S$ be a generic Enriques surface in the sense of \cite{BP}. Then the whole automorphism 
group $\Aut (S)$ acts on $S$ semi-symplectically.
More generally, let $S$ be an Enriques surface whose covering $K3$ surface $X$ has the 
following genericity property:
\begin{itemize}
\item The transcendental lattice
%\footnote{It is the orthogonal complement of $NS(X)$ inside the second cohomology $H^2(X,\mathbb{Z})$.}
$T_X$ of $X$, considered as a lattice equipped with Hodge structure,
has only automorphisms $\{\pm \id_{T_X}\}$.
\end{itemize}
Then the whole group $\Aut (S)$ acts on $S$ semi-symplectically. 
For example, this is the case when the Picard number $\rho (X)$ is odd,
since the value of the Euler function $\phi (n)$ is an even number for all $n\geq 3$
and by \cite{nikulin-auto}.
\end{example}

\subsection{Semi-symplectic automorphisms of finite order}
From now on, we study %finite\footnote{means the order is finite}
automorphisms of Enriques surfaces of finite order, 
which we simply call finite automorphisms as a matter of convenience.  
First we have the following criterion.

\begin{prop}\label{2356}
Let $\sigma$ be a finite automorphism of an Enriques surface $S$. If $\ord (\sigma )$ is not
divisible by $4$, then it is automatically semi-symplectic. %For example, automorphisms of order $2,3,5,$ or $6$ are always semi-symplectic.
%\begin{enumerate}
%\item If $\ord (\sigma) =2$ then it is semi-symplectic.
%\item If $\ord (\sigma)$ is odd, then it is semi-symplectic.
%\item If $\ord (\sigma)$ is $2p$ where $p$ is an odd number, then it is semi-symplectic.
%\end{enumerate} 
\end{prop}
\begin{proof}
We separately argue two cases (1) $\ord (\sigma)=2$ and (2) $\ord (\sigma)$ is odd.
The rest is easily deduced.\\
(1) Let $\varphi$ be one of two lifts to $X$.
Then $\varphi^2$ is a lift of $\id_S$ hence is either $\id_X$ or $\varepsilon$. 
In the former case, either $\varphi$ or $\varphi \varepsilon$ is symplectic on $X$ and we are done 
by Proposition \ref{ssp}. 
Suppose the latter occurs. Then $\varphi$ is an automorphism of order $4$ on $X$ 
which has no nontrivial stabilizer subgroup over $X$. Thus $X\rightarrow X/\varphi =:Y$
is an \'{e}tale covering of degree $4$. But then $\chi (\mathcal{O}_{Y})=
\chi (\mathcal{O}_X)/4=1/2$ must be an integer, a contradiction. \\
(2) Let $n$ be the order of $\sigma$. Arguing as in (1), we see that the two lifts of $\sigma$ 
have orders $n$ and $2n$. We denote by $\varphi$ the one with order $n$ and we show
that it is symplectic. Suppose that $\varphi$ is non-symplectic. 
Then $Y:=X/\varphi$ is a rational surface with at most quotient singularities.
The free involution $\varepsilon$ induces an involution $\varepsilon_Y$ on $Y$, which
must have a fixed point $P\in Y$ because $\chi (\mathcal{O}_{Y})=1$. 
Let $\pi_Y \colon X\rightarrow Y$ be the quotient morphism. Then the fiber $\pi_Y^{-1} (P)$ consists of 
odd number of points and has an action by $\varepsilon$ by construction. 
However this is impossible since $\varepsilon$ is fixed-point-free of order $2$.
Thus $\varphi$ is symplectic.
%(3) Finally if the order is $2n$ with $n$ odd, then in the decomposition $\sigma=(\sigma^n)(\sigma^{n+1})$,
%$\sigma^n$ has order $2$ and $\sigma^{n+1}$ has odd order. Therefore this case follows from the previous 
%discussions.
\end{proof}

\begin{rmk} The case (2) in the proof above also shows that if $\sigma$ is an automorphism of odd order, 
then $S/\sigma\simeq X/\langle \varphi, \varepsilon\rangle$ is an Enriques surface (with singularities in general).
\end{rmk}

%\begin{example}\label{nonss}
%Here we remark that there exists a non-semi-symplectic automorphism of order $4$.
%Let $(S,\sigma)$ be the numerically trivial automorphism of order $4$ from \cite[Section 3]{MN}.
%It is shown there that for a general such $S$, the covering $K3$ surface $X$ has the transcendental lattice
%$T_X\simeq U\oplus U(2)$ and a suitable lift $\varphi$ of $\sigma$ 
%acts by the $(4,4)$-matrix $A_0$ (\cite[p390]{MN}) on $T_X$. 
%The eigenpolynomial of this matrix $A_0$ is $(\lambda^2+1)^2$; thus $\varphi$ is 
%properly non-symplectic of order four on $X$ and $\sigma$ acts on the space 
%$H^0 (S, \mathcal{O}_S(2K_S))$ by a negation. 
%\end{example}

By \cite{nikulin-auto, mukai88}, a finite symplectic automorphism $\varphi$ of a $K3$ surface
has  $\ord (\varphi)\leq 8$ and the following table holds.
\begin{equation}\label{m24}%\tag{$\ast\ast\ast$}
\begin{tabular}{c|ccccccc}
{\rm order of }$\varphi$ &2&3&4&5&6&7&8\\ \hline
{\rm number of fixed points} &8&6&4&4&2&3&2
\end{tabular}
\end{equation}
\begin{cor}\label{atmost6}
A finite semi-symplectic automorphism $\sigma$ has $\mathrm{ord}(\sigma)\leq 6$.
\end{cor}
\begin{proof}
By Corollary \ref{slift}, $\sigma$ has a symplectic lift $\varphi$ of the same order, 
hence the order is at most $8$.
If $\mathrm{ord}(\varphi)=7$, then $\varepsilon$ cannot act freely on the 
fixed point set $\Fix(\varphi)$ by \eqref{m24}, a contradiction. If $\mathrm{ord}(\varphi)=8$, 
$\varphi$ has exactly two fixed points $P$ and $Q$ and they are exchanged by $\varepsilon$.
However, by applying the holomorphic Lefschetz formula, the local linearized actions 
$(d\varphi)_P, (d\varphi)_Q$ are given by 
\[\begin{pmatrix} \zeta_8 & 0 \\ 0 & \zeta_8^7\end{pmatrix} \text{ and }
\begin{pmatrix} \zeta_8^3 & 0 \\ 0 & \zeta_8^5\end{pmatrix}\ 
\text{ where }\zeta_8=e^{2\pi\sqrt{-1}/8}.
\]
These matrices are not conjugated by $d\varepsilon$ and we get a contradiction.
Therefore we see that $\mathrm{ord}(\sigma)\leq 6$.
\end{proof}

Next we want to look more closely at the fixed point set $\Fix(\sigma)$. 
Let $\varphi$ be the symplectic lift and $\varphi\varepsilon$ be the non-symplectic one. 
From the relation $S=X/\varepsilon$, $\Fix(\sigma)$ has the decomposition
$\Fix(\sigma)=\Fix^+(\sigma) \sqcup \Fix^-(\sigma),$ where $\Fix^+ (\sigma)= \Fix(\varphi)/\varepsilon$ is the set of {\em{symplectic}} fixed points and 
$\Fix^- (\sigma)=\Fix(\varphi\varepsilon)/\varepsilon$ is the set of 
{\em{anti-symplectic}} fixed points. Geometrically, they are distinguished by 
the determinant ($=\pm 1$) of the local linearized action $(d\sigma)_P$. 

The number of symplectic fixed points is deduced from the table \eqref{m24}
by $\#\Fix^+ (\sigma)=\#\Fix (\varphi)/2$.
On the other hand, the anti-symplectic fixed point set $\Fix^-(\sigma)$ has a variation.
Here we use the topological Lefschetz formula 
\begin{equation}\label{lefschetz}%\tag{$\ast\ast\ast\ast$}
\sum_{i=0}^4 (-1)^i \tr (\sigma^*\mid_{H^i(S, \mathbb{Q})})=\e (\Fix (\sigma )).
\end{equation}
to give a rough classification\footnote{For a detailed classification 
of involutions, see \cite{IO}.}.
The quantity (\ref{lefschetz}) is called the {\em{Lefschetz number}} and 
denoted by $L(\sigma)$.

\begin{prop}\label{chara}
Let $\sigma$ be a semi-symplectic automorphism of order $n\leq 6$.
\begin{itemize}
\item If $n=2$, then $\Fix^- (\sigma)$ is a disjoint union of smooth curves.  
The Lefschetz number $L(\sigma)$ takes (every) even number from $-4$ to $12$. 
\item If $n=3$, then $\Fix^- (\sigma)= \emptyset$ and we have $L(\sigma)=3$. 
\item If $n=4$, then $\Fix^- (\sigma)$ is either empty or $2$ points. Accordingly 
$L(\sigma)$ equals either $2$ or $4$.
\item If $n=5$, then $\Fix^- (\sigma)= \emptyset$ and we have $L(\sigma)=2$. 
\item If $n=6$, then $\Fix^- (\sigma)$ is either empty or $2$ points. Accordingly 
$L(\sigma)$ equals either $1$ or $3$.
\end{itemize}
\end{prop}
\begin{proof}
We have only to compute $\Fix^- (\sigma)$ in each case.

($n=2$) The first assertion follows since the local action $(d \sigma)_P$ at $P\in \Fix^- (\sigma)$
is of the form $\diag (1,-1)$.
On the other hand, the action $\sigma^* \curvearrowright H^2 (S, \mathbb{Q})$
is identified with $\varphi^* \curvearrowright H^2(X, \mathbb{Q})_{\varepsilon^* =1}$, 
the invariant subspace with respect to $\varepsilon^*$.
It is known (also easily deduced from Table (\ref{m24}) 
via Lefschetz formula) that 
the action $\varphi^*$ has an $8$-dimensional negative eigenspace on $H^2(X,\Q)$.
Therefore, by counting the eigenvalues, we get the assertion (2). 
The existence of every value is shown in \cite{IO}.

($n=3,5$)  If the order $n$ is odd, $((d\sigma)_P)^n\neq 1$ for $P\in \Fix^-(\sigma)$. 
Therefore there are no anti-symplectic points.

($n=4,6$) Since $(\varphi \varepsilon)^2=\varphi^2$, we have
$\Fix (\varphi \varepsilon)\subset \Fix (\varphi^2)$. Since $\varepsilon$ is free, we get 
$\Fix (\varphi \varepsilon)\subset \Fix (\varphi^2)-\Fix (\varphi)$.
This latter set $T$ consists of $4$ points in both cases $n=4,6$ and
$\varphi$ and $\varepsilon$ both acts freely of order $2$ on $T$. 
Thus we see that the fixed points of $\varphi \varepsilon$ are either whole $T$ or empty.
\end{proof}

\subsection{Mathieu automorphisms}\label{ll}

The left-hand-side of \eqref{lefschetz} can be regarded as the character of the representation
$\sigma^* \curvearrowright H^*(S,\Q)$ since odd dimensional (rational) cohomology vanishes.
In contrast to symplectic automorphisms of $K3$ surfaces, 
Proposition \ref{chara} shows that we cannot treat semi-symplectic automorphisms 
uniformly from the viewpoint of characters. Nevertheless, we can make the following 
definition.

\begin{dfn}\label{Maction}
\begin{enumerate}
\item[(1)] A $12$-dimensional representation $V$ of a finite group $G$ over a field of characteristic zero is called a {\em{small Mathieu representation}} 
if its character $\mu (g)$ depends only on $\ord (g)$ and 
coincides with that of (the permutation representation on $\Omega_+$ of)
the small Mathieu group $M_{11}$.
\item[(2)] Let $G$ be a finite group of automorphisms of an Enriques surface $S$.
The action is called {\em{Mathieu}} if it is semi-symplectic and the induced representation 
$G \curvearrowright H^*(S,\Q)$ is a small Mathieu representation.
\end{enumerate}
\end{dfn}
For the characters of $M_{11}$, see ($\ast\ast$) in Introduction. 
We note that $G$ acts on $S$ effectively and without 
numerically trivial automorphisms if it is Mathieu. 
Comparing ($\ast\ast$) and Proposition \ref{chara}, we see that Mathieu 
condition has effects only
on elements of even orders. A little stronger statement holds as follows.
\begin{lem}\label{criterion}
A semi-symplectic group action of $G$ on an Enriques surface $S$ is Mathieu if for every element 
$\sigma$ of order $2$ or $4$, we have $L(\sigma)=4$.
\end{lem}
\begin{proof}
We prove that under the condition, any element $\sigma\in G$ of order $6$ 
have $L(\sigma)=1$. 
Let $a_k\ (k=0,\cdots, 5)$ be the number of the eigenvalue $e^{2\pi k\sqrt{-1}/6}$
in the representation 
$\sigma^* \curvearrowright H^* (S, \mathbb{Q})$.
Since $\e (S)=12$ and the representation is over $\Q$, we have 
\[a_0 +\cdots +a_5=12,\quad a_1=a_5,\quad  a_2=a_4.\]
By assumption, we also have $L(\sigma^2)=3$ and $L(\sigma^3)=4$, which translates into
\[a_0-a_1-a_2+a_3=3,\quad a_0-2a_1+2a_2-a_3=4.\] 
On the other hand, by Proposition \ref{chara} (6), 
\[L(\sigma)=a_0+a_1-a_2-a_3=1\text{ or }3.\]
The only integer solution to these equations is given by $a_0=4, a_1=a_5=1,
a_2=a_4=2,a_3=2$. Therefore we get $L(\sigma)=1$. 
\end{proof}
By Proposition \ref{chara} and Definition \ref{Maction}, we have
%can give a precise description of $\Fix (\sigma)$ 
%for Mathieu automorphisms as follows.
\begin{prop}\label{local}
Let $\sigma$ be a Mathieu automorphism of an Enriques surface of order $n\geq 2$.
%We denote by $\zeta_n=e^{2\pi\sqrt{-1}/n}$ the primitive $n$-th root of unity.
Then the fixed locus $\Fix (\sigma)=\Fix^+(\sigma)\sqcup \Fix^-(\sigma)$ is as follows.
\begin{center}
\begin{tabular}{c|ccccc}
$n$ & $2$ & $3$ & $4$ & $5$ & $6$ \\ \hline
$\Fix^+(\sigma)$ & $4$ pts. & $3$ pts. & $2$ pts. & $2$ pts. & $1$ pts.\\ \hline
$\Fix^-(\sigma)$ & $\sqcup_i C_i$ & $\emptyset$ & $2$ pts. & $\emptyset$ & $\emptyset$
\end{tabular}
\end{center}
Here $\sqcup_i C_i$ is a disjoint union of smooth curves whose Euler number $\sum_i \e (C_i)$ is zero.
\end{prop}
%\newpage

\section{Finite groups with small Mathieu representations}\label{sMrep}

In this section we prove 
\begin{prop}\label{orders}
Let $G$ be a finite group which has a small Mathieu representation $V$ with character $\mu$.
Then the order of $G$ is
\[ 2^{a_2} 3^{a_3} 5^{a_5} 11^{a_{11}}\]
for non-negative integers $a_2\leq 4$, $a_3\leq 2$, $a_5, a_{11}\leq 1$.
Moreover, this bound is sharp since $M_{11}$ has the order 
$7920=2^4\cdot 3^2\cdot 5\cdot 11$.
\end{prop}

The definition of a small Mathieu representation is in Definition \ref{Maction}.
In what follows we use the notation $\mu(G)=(\sum_{g\in G} \mu (g))/|G|$.
From the theory of characters, we have $\dim V^G=\mu(G)$. In particular, 
\begin{equation}\label{dimension}
\mu(G) \text{is a non-negative integer.}
\end{equation}
Note that a subgroup $H$ inherits a small Mathieu representation and we have 
$\mu(H)\geq \mu(G)$. 
Also for a normal subgroup $N$, we can define the function $\mu$ on $G/N$ and 
we have $\mu (G/N)\geq \mu(G)$. The equality holds if and only if $N=\{1\}.$
%
%\begin{prop}\label{reptheory}
%Let $V$ be a small Mathieu representation of $G$ with character $\mu$. Then we have 
%\begin{(enumerate)}
%\item $\mu (g)$ is an integer for every $g\in G$.
%\item $(\psi, \mu )= \frac{1}{|G|} \sum_{g\in G}(\psi(g)\mu(g))$ is a non-negative integer 
%for every character $\psi$ of $G$.
%\item $\mu (G):=\frac{1}{|G|} \sum_{g\in G} \mu (g)$ is equal to $\dim V^G$. 
%In particular $\mu (G)$ is a non-negative integer.
%\end{(enumerate)}
%\end{prop}
\begin{proof} We begin the proof of Proposition \ref{orders}.
From Table ($\ast\ast$), every element $g\in G$ has $\ord (g)\in \{1,2,3,4,5,6,8,11\}$.
For a prime number $p$, let $G_p$ be a Sylow $p$-subgroup of $G$.
For odd $p$, $G_p$ does not contain elements of order $p^2$. Hence by \eqref{dimension},
\[\mu(G_p)=(12+e_p(|G_p|-1))/|G_p|\in \Z \quad (e_3=3,e_5=2,e_{11}=1)\]
shows that $|G_3|\leq 3^2, |G_5|\leq 5, |G_{11}|\leq 11$.

In the rest, let us replace $G$ by $G_2$ and show $|G|\leq 2^4$ for $p=2$. 
\begin{lem}\label{ab2^4}
Abelian groups of order $2^4$ have no small Mathieu representations.
\end{lem}
\begin{proof}
An abelian group of order $2^4$ is isomorphic to either 
$C_{16}$, $C_{8}\times C_2$, $C_4\times C_4$, $C_4\times C_2^2$ or $C_2^4$.
By definition $C_{16}$ has no small Mathieu representation. For other groups we can easily compute $\mu (G)$ 
to see that they don't satisfy the condition \eqref{dimension}. Hence we get the lemma.
\end{proof} 
Let $A$ be a maximal normal abelian subgroup of $G=G_2$.
We have $|A|\leq 2^3$ by Lemma \ref{ab2^4}.
%By Satz 7.3, Kapitel III of \cite{Hup}, if $|A|=2^a$ then $|G|\leq 2^{a(a+1)/2}$ holds. 
%Hence if $|A|\leq 2^2$ then we readily have $|G|\leq 2^3$. 
Since $G$ is a $2$-group, the centralizer $C_G(A)$ coincides with $A$ and 
the natural homomorphism $\varphi\colon G/A \rightarrow \Aut (A)$ is injective. 
Thus for $A\simeq C_2, C_4, C_2^2$, we have
$\Aut (A)\simeq \{1\}, C_2, \mathfrak{S}_3$ and we see that $|G|\leq 2^3$.

It remains to consider the case $|A|=2^3$. There are three abelian groups $A\simeq C_8$, $C_4\times C_2$ and $C_2^3$. We treat them separately.

\noindent {\bf{Case $A\simeq C_8$ :}}
Here we have $\Aut (A)\simeq C_2^2$. To show $|G|\leq 2^4$, it suffices to show that $\varphi$ is not 
surjective. Assume the contrary. 
Then there exists $x\in G$ which acts on the generator $g$ of $A$ by $xgx^{-1}=g^5$. We set 
$H=\langle A, x \rangle$. Since $x^2$ commutes with $A$, we get $x^2\in A$ and $|H|=2^4$.  
The equality $(g^i x)^4=g^{4i}x^4$ shows that $(g^i x)^4=1$ if ($x^4=1$ and $i$ is even) or ($x^4=g^4$ 
and $i$ is odd), and $\ord (g^i x)=8$ otherwise.
This enables us to compute
\[\mu (H) = \frac{1}{2} (\mu (A) + \mu (Ax)) = \frac{1}{2} (4+3) \not\in \mathbb{Z}.\]
Therefore $\varphi$ is not surjective.

\noindent {\bf{Case $A\simeq C_4\times C_2$ :}}
Let $g,h$ be generators of $A$ with
$\ord (g)=4$ and $\ord (h)=2$. Then $\Aut (A)$ is generated by 
\begin{equation*}
\begin{split}
\alpha & \colon \begin{pmatrix} g & h \end{pmatrix} \mapsto 
\begin{pmatrix} g+h & 2g+h \end{pmatrix},\\
\beta & \colon \begin{pmatrix} g & h \end{pmatrix} \mapsto 
\begin{pmatrix} g & 2g+h \end{pmatrix},
\end{split}
\end{equation*}
and is isomorphic to $D_8$. 
Assume that $|\mathrm{Im}\,(\varphi)|\geq 4$. Then $G$ has an element $x$ such that $\varphi (xA)=
\alpha^2$ and $x^2\in A$. By $(ax)^2=x^2 (a\in A)$, it follows that 
all elements in the coset $Ax$ have the same order $\ord (x)\in \{2,4,8\}$. 
Hence $\mu (Ax)$ is even, while $\mu (A)=5$. Thus the group $H=\langle A,x \rangle$ 
cannot have a small Mathieu representation by \eqref{dimension}, a contradiction.  
Hence $|\mathrm{Im}\,(\varphi)|\leq 2$.

\noindent {\bf{Case $A\simeq C_2^3$ :}}
We have $\Aut (A) =\GL (3, \mathbb{F}_2)$ which is the simple group of order $168$. 
(One of) its $2$-Sylow subgroups 
consist of elements 
\[\begin{pmatrix} 1 & \alpha & \beta \\
0 & 1 & \gamma \\ 0 & 0 & 1 \end{pmatrix},\quad  (\alpha, \beta, \gamma \in \mathbb{F}_2)\] 
and is isomorphic to $D_8$. 
Let us assume that the subgroup $\mathrm{Im}(\varphi)$ has at least four elements.  
Then $G$ contains an element $x$ whose image by $\varphi$ is conjugate to 
\[B=
\begin{pmatrix} 1 & 0 & 1 \\ 0 & 1 & 0 \\ 0 & 0 & 1 \end{pmatrix}=
\begin{pmatrix} 1 & 1 & 0 \\ 0 & 1 & 1 \\ 0 & 0 & 1 \end{pmatrix}^2 \]
and $x^2\in A$.
In suitable coordinates of $A$, it is easy to see that every element in the 
coset $Ax$ has order at most $4$. Hence we get the contradiction 
\[\mu (\langle A,x \rangle ) = \frac{1}{2} (\mu (A) + \mu (Ax)) = \frac{1}{2} (5+4) \not\in \mathbb{Z}.\]

Thus in all cases we get $|G|\leq 2^4$ and we obtain Proposition \ref{orders}.
\end{proof}

\section{Proofs of the Main Theorem}\label{equivalence}

We give the proofs to the main results, Theorems \ref{main} and \ref{main2}.
\begin{thm}\label{maintheorem}
The following conditions are equivalent to each other for a finite group $G$.
\begin{(enumerate)}
\item $G$ has a Mathieu action on some Enriques surface.
\item $G$ can be embedded into the symmetric group $\mathfrak S_6$ and the order  $|G|$  is not divisible by $2^4$.
\item $G$ is a subgroup of one of the following five {\em{maximal groups }}: 
\[ \mathfrak{A}_6, \mathfrak{S}_5, N_{72}=C_3^2\splitext D_8,C_2\times \mathfrak{A}_4, C_2\times C_4.\] 
\item $G$ is a group with a small Mathieu representation $V$ with $\dim V^G\geq 3$, 
whose $2$-Sylow subgroup is embeddable into $\mathfrak{S}_6$ and $G\not\simeq Q_{12}$.
\end{(enumerate)}
\end{thm}
\begin{rmk}\label{Q12}
(1) There are $25$ isomorphism classes of $G\neq \{ 1\}$ which satisfy the conditions (1)-(4) of the theorem. They are the groups exhibited in Propositions 
\ref{nonsolvable}, \ref{nilpotent} and \ref{solvable}.\\
(2) The group $Q_{12}$ has a
subtle behavior in the condition (4) of the theorem. Its $2$-Sylow subgroup is obviously embedded into 
$\mathfrak{S}_6$. Moreover, using the notation in \cite[Table 2]{nakamura}, the character
$4\chi_0+\chi_2+2\chi_3+\chi_4+\chi_5$ is a small 
Mathieu representation with $\dim V^{Q_{12}}=4$. 
But it has no Mathieu actions by Lemma \ref{12}.
It is thus necessary to put the extra condition on this group in the condition (4).
%Let $\mu_i\ (i=1,\dots, 6)$ be the irreducible characters of the representations $V_i$ defined by:
%\begin{equation*}
%\begin{split}
%V_1 &\colon g\mapsto 1,\quad  h\mapsto 1.\\
%V_2 &\colon g\mapsto 1,\quad  h\mapsto -1.\\
%V_3 &\colon g\mapsto -1,\quad  h\mapsto \sqrt{-1}.\\
%V_4 &\colon g\mapsto -1,\quad  h\mapsto -\sqrt{-1}.\\
%V_5 &\colon g\mapsto \begin{pmatrix} 0 & -1 \\ 1 & 1 \end{pmatrix},\quad h\mapsto \begin{pmatrix} \sqrt{-1} & 0 \\ -\sqrt{-1} & -\sqrt{-1} \end{pmatrix}. \\
%V_6 &\colon g\mapsto \begin{pmatrix} -1 & -1 \\ 1 & 0 \end{pmatrix},\quad h\mapsto \begin{pmatrix} -1 & 0 \\ 1 & 1 \end{pmatrix}. \\
%\end{split}
%\end{equation*}
%Then the character $4\mu_1+\mu_3+\mu_4+\mu_5+2\mu_6$ is small Mathieu. 
\end{rmk}
%\begin{proof}
We have already shown (3)$\Rightarrow$(1) in Section \ref{3->1}. 
In the following three subsections we prove the rest in the order (1)$\Rightarrow$(4) $\Rightarrow$(2) $\Rightarrow$(3).
%This finishes the proof of the main theorem.
%\end{proof}

\subsection{Proof of (1)$\Rightarrow$(4)}\label{1to4}
%(1)$\Rightarrow$(4)\ 
By definition, $H^*(S,\Q)$ is a small Mathieu representation.
Obviously $H^i(S,\mathbb{Q})\,(i=0,4)$ are invariant subspaces and for 
any ample divisor $H$ on $S$ the sum $H'=\sum_{g\in G} gH$ is a $G$-invariant ample divisor, 
showing $H^2(S,\mathbb{Q})^G\neq 0$. Putting together, 
we find $\dim H^*(S,\mathbb{Q})^G\geq 3$.

Next we show that the 2-Sylow subgroup $G_2$
is embeddable in $\mathfrak{S}_6$. By Corollary \ref{atmost6} every element $g\in G_2$ has 
$\ord (g)\leq 4$. By the character table ($\ast\ast$), we see that $\mu (g)=4$ unless $g=1$. 
Thus the condition \eqref{dimension},% (Section \ref{sMrep}),
\[\mu (G_2)=\frac{1}{|G_2|}(12+4(|G_2|-1))\in \mathbb{Z}\]
gives $|G_2|\leq 2^3$.
It is easy to check that all $2$-groups of order at most $8$,
except for $G_2=C_8$ and $Q_8$, can be embedded in the 
group $C_2\times D_8$, the $2$-Sylow subgroup of $\mathfrak{S}_6$.
The cyclic group $C_8$ is impossible by Corollary \ref{atmost6}.
The group $Q_8$ is also impossible by the following lemma, which concludes 
$G_2\subset \mathfrak{S}_6$.
\begin{lem}\label{Q8}
No Mathieu actions on Enriques surfaces by the quarternion group, 
$Q_8= \bracket{g,h\mid g^4=1, hgh^{-1}=g^{-1}, g^2=h^2}$, 
exist.
\end{lem}
\begin{proof} 
By means of contradiction, suppose that we had one. We denote by $i=g^2$ the unique and central involution in $Q_8$. 
By Proposition \ref{local}, we see that the fixed point sets of
$g$ and $h$
both coincide with the set of four isolated fixed points of $i$. In particular, 
these four points are fixed by the whole group.

Let $P$ be one of anti-symplectic fixed points of $g$, which exists by Proposition \ref{local}.
By looking at differentials, we obtain 
a map $d_P \colon Q_8\rightarrow \GL (T_P S)$, which is injective by 
the complete reducibility for finite groups.
But since any embedding of $Q_8$ into $\GL (2,\mathbb{C})$ 
factors through $\mathrm{SL} (2,\mathbb{C})$, this contradicts to that $P$ was an anti-symplectic fixed point. 
This proves the lemma.
(Proof of the latter fact: Note that the diagonal form of the involution $d_P (i)$ is either $\diag (1,-1)$ or 
$\diag (-1,-1)$. In the former case, its centralizer in $\GL (2,\mathbb{C})$ is the 
commutative group of diagonal matrices, 
hence we get a contradiction. 
In the latter case, from $g^2=h^2=i$ and $gh\neq hg$, we must have that both $d_P(g)$ and 
$d_P(h)$ have $\tr=0$ and $\det =1$.) 
\end{proof}
Finally we show $G\not\simeq Q_{12}$.
\begin{lem}\label{12}
No Mathieu actions on Enriques surfaces by the group $Q_{12}=\langle g,h\mid g^6=1, h^2=g^3, hgh^{-1}=g^{-1}\rangle $ exist.
\end{lem}
\begin{proof}
Assume we had one.
Since $g^3$ is the unique involution in $Q_{12}$, by Proposition \ref{local}, we see that the fixed point 
sets of $h$ and $gh$ both coincide with the set of four isolated fixed points of $g^3$. 
We denote them by $P_i\ (i=1,\dots, 4)$. However from the equations $h(P_i)=P_i, gh(P_i)=P_i$ we get 
$g(P_i)=P_i$, contradicting to that $g$ of order $6$ has a unique isolated fixed point by 
Proposition \ref{local}.
\end{proof}

\subsection{Proof of (4)$\Rightarrow$(2)}\label{4to2}

\begin{lem}
Let $G$ have a small Mathieu representation and assume that $\mu (G)\geq 3$.
Then $11$ does not divide $|G|$.  
\end{lem}
\begin{proof}
If $G$ has a nontrivial $11$-Sylow subgroup $G_{11}$, the dimensions of invariant subspaces satisfy 
$2=\mu (G_{11})\geq \mu (G)$,
a contradiction.
\end{proof}
\begin{lem}\label{2gp}
Let $G$ have a small Mathieu representation and assume that $G_2$ is embeddable into $\mathfrak{S}_6$.
Then $G$ has no elements of order $8$ and $|G_2|\leq 2^3$. 
%We have the equality if and only if $G_2$ is isomorphic either to
%$C_2^3, C_2\times C_4$ or $D_8$. In particular there are no elements of order $8$.
\end{lem}
\begin{proof}
Since the $2$-Sylow subgroup $(\mathfrak{S}_6)_2$ is isomorphic to $C_2\times D_8$, 
$G_2$ has no elements of order $8$. 
By the definition of small Mathieu character $\mu$, 
every non-identity element $g\in G_2$ has character $\mu (g)=4$. 
Thus the condition 
\[\mu (G_2) = \frac{1}{|G_2|} (12+4(|G_2|-1))\in \mathbb{Z}\]
gives $|G_2|\leq 8$. 
%Among the five groups of order $8$, we easily see that $C_8$ and $Q_8$
%have no embeddings into $C_2\times D_8$. Therefore our assertion holds.
\end{proof}
\begin{cor}\label{order}
Let $G$ be a finite group that satisfies the condition (4) of Theorem \ref{maintheorem}. 
Then for all $g\in G$ we have $\ord(g)\leq 6$. Moreover we have 
\[|G|= 2^{a_2} 3^{a_3} 5^{a_5}\] 
for non-negative integers $a_2\leq 3, a_3\leq 2, a_5\leq 1$. 
\end{cor}
\begin{proof} This follows immediately by combining Proposition \ref{orders} and lemmas above.
\end{proof}
In particular, we have proved the latter part of (2) of Theorem \ref{maintheorem}.

In the following, we classify all groups that satisfy the condition (4). 
%After seeing all possible groups, it is easy to show that they in fact 
%sit inside $\mathfrak{S}_6$. This part is left to the reader.
First we consider non-solvable groups.
Recall that $G$ is non-solvable if and only if at least one of its composition factors is 
a non-abelian finite simple group.
\begin{prop}\label{nonsolvable} 
Let $G$ be a finite group that satisfies the condition (4) of Theorem \ref{maintheorem}.
Assume that $G$ is non-solvable. Then $G$ is isomorphic either to $\mathfrak{A}_5, \mathfrak{S}_5$ or 
$\mathfrak{A}_6$.
\end{prop}
\begin{proof}
Let $N$ be a composition factor of $G$ which is a non-abelian simple group.
By Corollary \ref{order} and the table of finite simple groups (see \cite{ATLAS}), $N$ is 
either $\mathfrak{A}_5$ or $\mathfrak{A}_6$. In the latter case we see $G=N\simeq \mathfrak{A}_6$
by Corollary \ref{order}.

Let us continue the case $N\simeq \mathfrak{A}_5$. By order reason, $N$ is the only non-abelian simple factor.
We have subgroups $H\subset G$ and $T\triangleleft H$ such that $H/T\simeq \mathfrak{A}_5$. 
Since $H$ inherits the small Mathieu representation and $\mathfrak{A}_5$ is a quotient of $H$,
we have
\[3=\mu (\mathfrak{A}_5)=\mu (H/T)\geq \mu (H) \geq \mu(G)\]
by the discussion after Proposition \ref{orders}.
By the condition $\dim V^G\geq 3$, 
we have equalities. It follows that $T$ is trivial and $H\simeq \mathfrak{A}_5$. 

By Corollary \ref{order}, the index $[G:H]$ is a divisor of $6$. Hence the composition series 
of $G$ looks either (i) $H\triangleleft G$ or (ii) $H\triangleleft G'\triangleleft G$ (if it has more than
one terms). 
Let us begin with (i).
We consider the natural homomorphism $\varphi \colon G\rightarrow \Aut (H)$.
Since $G$ does not contain elements of order $5n\, (n\geq 2)$ and $H$ is center-free, 
we see that $\varphi$ is injective into $\Aut (\mathfrak{A}_5)\simeq \mathfrak{S}_5$.
Therefore $G$ is isomorphic to $\mathfrak{S}_5$. 
In the case (ii), we get $G'\simeq \mathfrak{S}_5$ by (i). We again consider the natural homomorphism
$\psi \colon G\rightarrow \Aut (G')$. By the same reasoning as before, $\psi$ is injective. 
But this is not the case since $G\supset G'\simeq \mathfrak{S}_5$ is a proper inclusion.

Thus we obtain the classification of non-solvable groups.
\end{proof}

We recall that $G$ is nilpotent if and only if $G$ is the direct product of its Sylow subgroups.
\begin{prop}\label{nilpotent}
Let $G$ be a nilpotent group that satisfies the condition (4) of Theorem \ref{maintheorem}.
Then $G\simeq C_n^a$ ($2\leq n\leq 6$ if $a=1$ and otherwise $(n,a)=(2,2),(3,2),(2,3)$), $C_2\times C_4$ or $D_8$.
\end{prop}
\begin{proof}
Corollary \ref{order} and the given condition on $2$-Sylow subgroups classify
the Sylow subgroups of $G$ as follows: 
$G_2$ is isomorphic to $C_2^3, C_2\times C_4, D_8$ or has order at most $4$, 
$G_3$ is isomorphic to $C_3^2$ or $ C_3$ and $G_5$ is isomorphic to $C_5$ (if not trivial). 

We claim that neither groups $H_1=C_2\times C_3^2$ nor $H_2=C_2^2\times C_3$ have small Mathieu representations.
In fact, the former has $\mu (H_1)=8/3\not\in \mathbb{Z}$. For the latter, 
let us choose generators $g,h\in H_2$ with $g^6=h^2=ghg^{-1}h^{-1}=1$. Let $\psi$ be the character of $H_2$
assigning $h\mapsto 1$ and $g\mapsto \zeta_6$, where $\zeta_6$ is the primitive $6$-th root of unity. 
Then the inner product of characters $(\psi, \mu)$ is $1/2\not\in \mathbb{Z}$. Thus $H_2$ 
does not have small Mathieu representations.

Recall that all elements $g\in G$ have $\ord (g)\leq 6$ by Corollary \ref{order}.
This fact with the non-existence of subgroups $H_i$ above leads us to the list. 
\end{proof}

Finally we treat the case $G$ is non-nilpotent and solvable.
Recall that any finite group has the maximal normal nilpotent subgroup $F$, called the 
{\em{Fitting subgroup}}. When $G$ is non-nilpotent and solvable, $F$ is a proper subgroup and 
it is known that the centralizer $C_G (F)$ coincides with the center $Z(F)$ of $F$. 
In particular the natural homomorphism $\varphi \colon G/F \rightarrow \Out (F)$ is injective,
where $\Out (F)=\Aut (F)/\mathrm{Inn} (F)$ is the group of outer automorphism classes. 
Moreover, by the extended Sylow's theorem for solvable groups, the exact sequence
\begin{equation}\label{F}
\begin{CD}
 1 @>>> F @>>> G @>>> G/F @>>> 1
\end{CD}
\end{equation}
splits if $|F|$ and $|G/F|$ are coprime. 
\begin{prop}\label{solvable}
Let $G$ be a non-nilpotent and solvable group that satisfies the condition (4) of Theorem \ref{maintheorem}.
Then $G$ belongs to the following list.
\[
\begin{tabular}{c|cccccc}
$G$ & $D_6$ & $D_{10}$ & $D_{12}$ & $\mathfrak{A}_4$ & $\mathfrak{A}_{3,3}$ & $C_3 \times \mathfrak{S}_3$\\ \hline
$|G|$ & $6$ & $10$ & $12$ & $12$ & $18$ & $18$ \\
\end{tabular}
\]
\[
\begin{tabular}{c|ccccccc}
$G$ &  $\mathrm{Hol}(C_5)$ & $C_2\times \mathfrak{A}_4$ & $\mathfrak{S}_4$ & $C_3^2\rtimes C_4$ & $\mathfrak{S}_{3,3}$ & $N_{72}$ \\ \hline
$|G|$ & $20$ & $24$ & $24$ & $36$ & $36$ & $72$ \\
\end{tabular}
\]
\end{prop}
\begin{proof}
Let $F$ be the Fitting subgroup. The isomorphism class of $F$ belongs to the list of Proposition \ref{nilpotent},
so we give separate considerations. We note that $\Out (F)=\Aut (F)$ if $F$ is abelian.

\noindent {\bf{Case: $F$ is cyclic }} In the table below, $-1$ denotes the inversion $g\mapsto g^{-1}\, (g\in F)$.
\begin{center}\vspace{-2mm}
\begin{tabular}{c|c|c|c|c|c}
$F$        &$C_2$&$C_3$&$C_4$&$C_5$&$C_6$ \\ \hline
$\Aut (F)$ &$\{ 1\}$ &$\{\pm 1\}$ &$\{\pm 1\}$ & $C_4$&$\{\pm 1\}$ \\ \hline
$G$        &  & $D_6$ & & $D_{10}$, $\mathrm{Hol}(C_5)$ & $D_{12}$ \\
\end{tabular}
\end{center}\vspace{-2mm}
For $F=C_2,C_4$ all extensions \eqref{F} are nilpotent by order reasoning.
For $F=C_3,C_5$, \eqref{F} splits and we get the table. Here $\mathrm{Hol}(C_5)$ denotes
the holomorph $C_5\rtimes \Aut (C_5)$.
For $F=C_6$, we get a split extension $D_{12}$. the other non-split extension $Q_{12}$ is 
not allowed by the assumption.

\noindent {\bf{Case: $F=C_2^2$ }} 
We have $\Aut (F)=\mathfrak{S}_3$. 
Since $G$ is non-nilpotent, it has an element $x$ of order $3$. 
Then $F$ and $x$ generate a subgroup $H$ isomorphic to $\mathfrak{A}_4$. 
If further the inclusion $H\subset G$ is proper, $H$ has index two and is normal. 
In this case, such $G$ is isomorphic to $\mathfrak{S}_4$ or $C_2\times 
\mathfrak{A}_4$, but in the latter group $C_2^2$ is not the Fitting subgroup.

\noindent {\bf{Case: $F=C_3^2$ }}
We have $\Aut (F)\simeq \GL (2,\mathbb{F}_3)$ which has order $48=2^4 3$.
This group has the semi-dihedral group $SD_{16}$ as its $2$-Sylow subgroup,
\[SD_{16}=\langle g,x\mid g^8=x^2=1, xgx^{-1}=g^3\rangle.\]
By Corollary \ref{order}, the extension \eqref{F} splits and $G/F$ is a subgroup of $SD_{16}$. 
Since the maximal subgroups of $SD_{16}$ are $C_8, D_8, Q_8$ and all are characteristic,
we see that isomorphic subgroups of order $8$ in $\GL (2,\mathbb{F}_3)$ are conjugate.
In view of Proposition \ref{nilpotent}, we get the unique extension
$G\simeq C_3^2\rtimes D_8=N_{72}$ if $|G|$ is maximal.
If $|G/F|=4$, we get unique extensions $C_3^2\rtimes C_4, C_3^2\rtimes C_2^2\simeq \mathfrak{S}_{3,3}$.
If $|G/F|=2$, we have two extensions which are isomorphic to $\mathfrak{A}_{3,3}, C_3\times \mathfrak{S}_3$. 
%We note that, all these are subgroups of $N_{72}$ by the construction.

\noindent {\bf{Case: $F=C_2^3$ }}
We have $\Aut (F)=\GL(3,\mathbb{F}_2)$, which is the simple group of order $168=2^3\cdot 3\cdot 7$. 
By Corollary \ref{order}, (\ref{F}) splits with $G/F\simeq C_3$. Since $3$-Sylow 
subgroups in $\GL(3,\mathbb{F}_2)$ are conjugate, 
we get the unique extension $G\simeq C_2\times \mathfrak{A}_4$. 

\noindent {\bf{Case: $F=C_2\times C_4,\,D_8$ }}
In these cases we have $\Aut (F)\simeq D_8$, hence 
we get no non-nilpotent groups. 

This completes the classification of possible groups. It is not difficult to check 
that every groups are in $\mathfrak{S}_6$, and 
we have proved our theorem.
\end{proof}

%(4)$\Rightarrow$(2)\
%We prove this part by extending the discussions in Section \ref{sMrep}, see the next subsection.\\

\subsection{Proof of (2)$\Rightarrow$(3)}\label{2to3}
%(2)$\Rightarrow$(3)\ 
By condition, $G\neq \mathfrak{S}_6$. By \cite{ATLAS}, $\mathfrak{S}_6$ has 
four isomorphism classes of maximal subgroups 
$\mathfrak{A}_6, \mathfrak{S}_5, N_{72}, C_2\times \mathfrak{S}_4$, the first three 
of which readily satisfy (3). 
Thus we may assume $G\subset C_2\times \mathfrak{S}_4$. 
Again by the order condition $G$ is a proper subgroup. 
A standard argument shows that $C_2\times \mathfrak{S}_4$ has 
maximal subgroups $\mathfrak{S}_4, C_2\times \mathfrak{S}_3, C_2\times \mathfrak{A}_4, C_2\times 
D_8$. The first two are subgroups of $\mathfrak{S}_5$. 
The third one is in the list of (3).
Thus we may assume $G\subset C_2\times D_8$.
Again by the order condition $G$ is a proper subgroup, and the maximal subgroups of $C_2\times D_8$ are 
$C_2^3, D_8, C_2\times C_4$. The first two groups are subgroups of $C_2\times \mathfrak{A}_4$
and $N_{72}$ respectively.
The last group $C_2 \times C_4$ is the final ingredient in (3), so our assertion holds.\\

%%%%%%%%%%%%%%%%%%%%%%%%%%%%%%%%%%%%%%%%%%%%%%%%%%%%%%%

\section{Tame automorphisms in positive characteristic}\label{pc}

Let $k$ be an algebraically closed field of positive characteristic $p>0$. 
Recall that an  Enriques surface $S$ over $k$ is characterized by the 
numerical equivalence $K_S\equiv 0$ and the second Betti number $b_{2}(S)=10$, including $p=2$.
Let $G$ be a finite group of tame automorphisms acting on $S$. 
Since $\dim H^0(S,\mathcal{O}_S(2K_S))=1$ and $\dim H^*(S,\Q_l)=12\ (l\neq p)$, 
the same definitions as Definition \ref{semi-sympl}
and \ref{Maction} (2) make sense. Hence we can speak of Mathieu automorphisms over $k$ as well.
In this section we use a result of Serre \cite{serre} to prove Theorem~\ref{main3}.
First we remark the following.

\begin{prop}\label{positive_char2}
In characteristic $p\geq 11$, any semi-symplectic finite group action on an Enriques surface is tame.
\end{prop}
%We begin with the proof of the second proposition.
\begin{proof}
Assume $p=\mathrm{char}\ k\geq 11$ and  %Suppose 
$\sigma$ is an automorphism of order $p$
of an Enriques surface $S$ over $k$. Then it lifts to the canonical $K3$-cover $X$
and commutes with the covering involution $\varepsilon$. 
The lift, which we denote by $\tilde\sigma$, is a wild automorphism of a $K3$ surface,
hence by \cite[Theorem 2.1]{DK09}, $p= 11$ follows.
Moreover, by \cite[Lemma 2.3]{order11}, $\dim H^2_{\mathrm{et}}(X,\Q_l)^{\sigma}=2$
for all $l\neq 11$. By choosing a prime $l$ so that the cyclotomic polynomial 
$T^{10}+\cdots+1$ is irreducible over $\Q_l$, this shows that the second cohomology 
is a sum of four irreducible modules of dimension 1, 1, 10, 10 over $\Q_l$. 
($l=2$ suffices, for example.)
But since $H^2_{\mathrm{et}}(X,\Q_l)^{\varepsilon}$ is 10 dimensional and contains 
a $\tilde\sigma$-invariant class, $\tilde\sigma$ and $\varepsilon$ cannot commute. 
%cannot be commutative.
Therefore we obtain a contradiction.
\end{proof}

%extend our classification to tame group actions. 
More explicitly Theorem~\ref{main3} is stated as follows.
%The results are the following Theorem \ref{positive_char} and Proposition \ref{positive_char2}.
\begin{thm}\label{positive_char}
A finite group $G$ has a tame Mathieu action on some Enriques surface $S$ over $k$ if and only if 
\begin{enumerate}
\item $G$ satisfies the same conditions as $(1)$-$(4)$ of Theorems \ref{main} and \ref{main2} when $\mathrm{char}\ k\geq 7$. 
\item $G$ is isomorphic to a subgroup of $N_{72}$, $\mathfrak{S}_4$, $C_2\times \mathfrak{A}_4$ or
$C_2\times C_4$ when $\mathrm{char}\ k=5$. 
\item $G$ is isomorphic to a subgroup of $\mathrm{Hol}(C_5)$, $D_8$, $C_2^3$ or $C_2\times C_4$ when
$\mathrm{char}\ k=3$.
\item $G$ is isomorphic to a subgroup of $C_3^2$ or $C_5$ when $\mathrm{char}\ k=2$. 
%(and $G$ is nontrivial).
\end{enumerate}
\end{thm}

\subsection{Theorem \ref{positive_char} in odd characteristics}
%\noindent {\bf{Proof of `only if' part of Theorem \ref{positive_char} for odd characteristics.}}

First we prove the `only if' part of Theorem \ref{positive_char}. 
Serre's theorem \cite[Theoreme 5.1]{serre} says that 
if $X$ is a smooth projective variety over $k$ with a tame automorphism group $G$ with
$H^2(X,\mathcal{O}_X)=H^2(X,\Theta_X)=0$, then $(X,G)$ lifts to characteristic zero.
If $\mathrm{char}\ k\geq 3$ and $S$ is an Enriques surface over $k$, 
%or $\mathrm{char}\ k=2$ and $S$ is classical, 
the assumptions are satisfied and hence $(S, G)$ lifts smoothly to characteristic zero.
By our result over $\C$, we see that $G$ is one of the $25$ groups of Theorems \ref{main} and \ref{main2}. By the condition of tameness, we can deduce the classification of 
groups from Section \ref{equivalence}. This gives the proof of `only if' part of Theorem 
\ref{positive_char} for odd characteristics.

Now we discuss the reductions modulo $p$ of our maximal group actions.

\begin{enumerate}
\item Example \ref{Ohashi} degenerates in characteristic $p=2,5$ since $\varepsilon$ 
becomes to have 
a fixed point $(1:1:1:1:1)\in \P^4$. In other characteristics, the same 
equation defines an Enriques surface and the group $\mathfrak{S}_5$ acts in the same way.
\item The surface \eqref{Hesse-Godeaux} with $\lambda,\mu =1\pm \sqrt{3}$
becomes reducible in characteristics $p=2,3$. 
In fact, it contains the 2-plane $x_i=y_i$, $i=0,1,2$.
%the surface becomes reducible $(p=2)$ or $\varepsilon$ becomes to have fixed points, $(1:1:1:0:0:0)$ for example $(p=3)$.
In other characteristics, we can check that the surface $X$ (with
$\lambda,\mu=1\pm \sqrt{3}$) is smooth and $\varepsilon$ has no fixed points. 
Moreover, the studies in Subsections \ref{subsection:N72} and \ref{subsection:A6}
both hold true without changes. Therefore the Mathieu actions by $N_{72}$ and $\mathfrak{A}_6$ 
exist in all characteristics $p\geq 5$. 
\item The equation of Example \ref{H192} becomes reducible in $p=2$, 
but in other characteristics $p\geq 3$ the surface is smooth and the group action
by $H_{192}$ remains well-defined. Therefore we also have the Mathieu actions
by $C_2\times C_4$ and $C_2\times \mathfrak{A}_4$ in characteristic $p\geq 3$.
\end{enumerate}

%\noindent {\bf{Proof of `if' part of Theorem \ref{positive_char} for odd characteristics.}}

Now we prove the `if' part.
If $\mathrm{char}\ k\geq 7$, then all the Examples \ref{Ohashi}, \ref{Mukai}, \ref{KOZex}
and \ref{H192} persist as Enriques surface. Therefore all groups satisfying the conditions of 
Theorem \ref{main} and \ref{main2} have Mathieu actions on some Enriques surface over $k$.
In $\mathrm{char}\ k=5$, since $\mathfrak{S}_4$ is a subgroup of $\mathfrak{A}_6$,
%maximal subgroups of $\mathfrak{A}_6$ whose order is coprime 
%to 5 are $\mathfrak{S}_4$ and $C_3^2\splitext C_4$, 
our result follows from the above study on Examples \ref{Mukai}, \ref{KOZex} and \ref{H192}. 
In $\mathrm{char}\ k=3$, since $\mathrm{Hol} (C_5)$ and $D_8$ are subgroups of $\mathfrak{S}_5$ and $C_2^3$ is contained in $C_2\times \mathfrak{A}_4$, 
%maximal subgroups of $\mathfrak{S}_5$ whose 
%order is coprime to 3 is $D_8$ or $\mathrm{Hol} (C_5)$, 
our result 
follows from the above study on Examples \ref{Ohashi} and \ref{H192}. 
This finishes the proof of Theorem \ref{positive_char} for odd characteristics.\\
%Note: in $p=2$, general member of Hesse-Godeaux has $C_3^2$. we have to check 
%that the multiplicative vector field $\delta$ has no zeros. $C_5$ is supplied by Kondo VI. 

\subsection{Theorem \ref{positive_char} in characteristic 2}
Finally we show Theorem \ref{positive_char} (4).  %in characteristic two. 
Since the conditions of Serre's theorem are often invalid, 
we give a direct treatment of the `only if' part modifying the proof of \cite[Proposition 2]{PSS}.

Let $\pi \colon S\rightarrow Y=S/\sigma$ be the quotient morphism 
by a tame semi-symplectic automorphism $\sigma$.
Let $P\in S$ be a fixed point of $\sigma$.
By the tameness, the action is locally linearizable at $P$ and we have the 
complete reducibility on $\mathcal{O}_{S,P}$ as in \cite[(1.1)]{mukai88}.
Therefore, $\pi (P)$ is a rational double point and $Y$ is smooth except those isolated singularities.  
Moreover, there is a nowhere vanishing global bi-canonical form on $Y$, 
%with nowhere zero except possibly at singular points, 
since $\sigma$ is semi-symplectic. 
In particular, we have $K_{\tilde Y}\equiv 0$ for the minimal resolution $\tilde{Y}$ of $Y$.
Hence we have
\begin{equation}\label{noether}
%0 \le 
c_2(\tilde{Y}) = 12 \chi(\mathcal O_{\tilde Y}) \le 24
\end{equation}
by Noether's formula.

\begin{lem}\label{no3,5}
Let $\sigma$ be a tame semi-symplectic automorphism in characteristic $2$ 
with order $q$, an odd prime number. Then one of the following holds. 
\begin{enumerate}
\item $\sigma$ is of order 3 and has three fixed points on $S$.
\item $\sigma$ is of order 5 and has two fixed points on $S$.
\end{enumerate}
In particular, $\sigma$ is automatically of Mathieu type.
\end{lem}
\begin{proof}
Let $\pi \colon S\rightarrow Y=S/\sigma$ and   $P\in \Fix\sigma$ be as above.
Then $\pi (P)$ is a rational double point of type $A_{q-1}$. 
%This shows that $Y$ is 
%a surface of Kodaira dimension zero and its minimal resolution $\tilde{Y}$ is either an
%abelian, (quasi-)hyperelliptic, $K3$ or Enriques surface.
%Putting $r$ to 
Denoting the number of fixed points by  $r$, we have 
\begin{equation}\label{euler}
12=c_2(S)=q(c_2(\tilde{Y})-qr)+r
\end{equation}
%where $c_2(\tilde{Y})$ is 0, 0, 24 or 12 respectively according to the classes above.
The integer solution of \eqref{euler}  exists in the range \eqref{noether}  only when $\tilde{Y}$ is an Enriques surface and
$(q, r) = (3,3), (5,2), (11,1)$. 
%$q=3,5,11$, $r=3,2,1$ respectively. 
But if $q=11$, then the exceptional curves of resolution of the singularity  span the negative definite lattice $A_{10}$ of rank 10 in $NS(\tilde{Y})$, 
a contradiction to $\rho(\tilde{Y})=10$. % and the Hodge index theorem. 
This shows the assertions (1) and (2). The last statement 
follows from the Lefschetz fixed point formula.
\end{proof}

\begin{lem}\label{no9}
There are no automorphisms of order 9, 15 or 25. Therefore 3 and 5 are the only 
possible orders.
\end{lem}
\begin{proof}
If $\sigma$ is of order $9$, then $\Fix(\sigma)$ is either empty or coincides with $\Fix(\sigma^3)$.
%$S^{\sigma}=\emptyset$ or $S^{\sigma}=S^{\sigma^3}$.
In the former case the equation corresponding to \eqref{euler} is 
\[12=9(c_2(\tilde{Y})-3)+3\]
and $c_2(\tilde{Y})=4$, which is impossible. In the latter case $Y$  has three rational double points of type $A_8$, whose rank is too large. 
Hence automorphisms of order 9 do not exist.
The other cases are treated in the same way.
\end{proof}

Now we are ready to prove the `only if' part.
Assume that a finite group $G$ has a tame action.
By Lemma~\ref{no3,5} and \ref{no9}, $G$ is of Mathieu type.
The discussion in Subsection \ref{4to2} is purely group-theoretic and remains 
true in our setting, too. Hence, we get the three groups by the lists in Propositions \ref{nonsolvable}, \ref{nilpotent} and \ref{solvable}.

Finally we show the existence of actions of $C_3^2$ and $C_5$.
\begin{enumerate}
\item Let $X=X_{\lambda,\mu}$ be the same as \eqref{Hesse-Godeaux} but we assume further that both $\lambda, \mu \in k$ are nonzero in characteristic 2.
%$\lambda \mu (\lambda + \mu)\neq 0$. 
Then $X$ has 12 nodes at the intersection of 12 conics defined in Section~\ref{section:Hesse-Godeaux}, for example at $(1:0:0:1:0:0)$ and $(1:1:1:\alpha:\alpha:\alpha)$, etc., with $\alpha=\sqrt{(1-\lambda)/(1-\mu)}$.
Moreover, $X$  is smooth elsewhere.
These nodes correspond to the 12 double points in fibers of the rational elliptic surface \eqref{Hesse cubic}. 
%considered in characteristic 2, , for example.
The usual involution 
$\varepsilon\colon (x:y)\mapsto (x:-y)$ of  $\mathbb P^5$
in characteristic $\ne 2$ is replaced 
by the action of the non-reduced group scheme 
$\boldsymbol{\mu}_2= \mathrm{Spec}\ k[t]/(t^2-1)$ defined by $(x:y)\mapsto (x:ty)$. %plays the same role as it. 
A local computation shows that 12 nodes disappear and the quotient $X/\boldsymbol{\mu}_2$ becomes a smooth (classical) Enriques 
surface. 
The group $C_3^2$ acts on $X/\boldsymbol{\mu}_2$ by \eqref{3^2}.
 
\item Let $X$ be the surface defined by the same equation as \eqref{VI} in characteristic 2. 
It has only 10 nodes and by taking the minimal resolution 
of the quotient by the Cremona
involution, we obtain a smooth (non-classical) Enriques surface with $\mathfrak{S}_5$-action.
In particular, it has a $C_5$ action.
\end{enumerate}
Thus the proof of Theorem~\ref{positive_char}(=Theorem~\ref{main3}) is completed.
%%%%%%%%%%%%%%%%%%%%%%%%%%%%%%%%%%%%%%%%%%%%%%%%%%%%%%%%%%%%%%%%
\appendix

\section{Lattice theoretic construction of Mathieu actions}\label{ltc}
We give a lattice theoretic proof of Theorem~\ref{veryMathieu}.

In \cite[Appendix]{mukai98},  for each  $G$  of the eleven groups  ($\ast$), 
a symplectic action on a K3 surface is constructed using 
\begin{(enumerate)}
\item the Niemeier lattice  $N$  of type $(A_1)^{24}$,
\item the action of the Mathieu group  $M_{24}$ on  $N$,
\item an embedding of $G$  into the Mathieu group  $M_{23}$,  and 
\item the Torelli type theorem for K3 surfaces.
\end{(enumerate)}
Here $N$  is even, unimodular and contains the root lattice 
\begin{equation}\label{A124}
\bigoplus_{i \in \Omega} \Z e_i,\quad (e_i^2)=-2
\end{equation}
as a sublattice of finite ($=2^{12}$) index, where $\Omega$  is the operator domain of  $M_{24}$.
The action $M_{24} \act \Omega$  extends (isometrically) on  $N$.
The key of the proof is to show the existence of a primitive embedding of  $N_G$  in the K3 lattice  $\Lambda \simeq 3U + 2E_8$,
where $N_G$  is the orthogonal complement of the invariant lattice  $N^G \subset N$.

In this appendix, making this construction $C_2$-equivariant, we give another proof of Theorem~\ref{veryMathieu}.
Namely we decompose \eqref{A124} in two parts
\begin{equation}
\bigoplus_{i \in \Omega_+} \Z e_i  \quad {\rm and} \quad \bigoplus_{i \in \Omega_-} \Z e_i,\quad (e_i^2)=-1,\ \forall i \in \Omega=\Omega_+ \sqcup \Omega_-
\end{equation}
with scaling by $1/2$.
Let  $N_{\pm}$  be the lattices obtained by  adding $(\sum_{i \in \Omega_\pm} e_i)/2$ to these.
$N_{\pm}$  is the dual of the root lattice of type  $D_{12}$, 
and $N_{\pm}(2)$  is an integral lattice of discriminant  $2^{10}$.

Let $G_{(6)}, G_{(9)}, G_{(10)} \subset M_{11}$ be the image of the embedding of the three groups $\mathfrak{S}_5$, $N_{72}$, $\mathfrak{A}_6$ in Lemma~\ref{embM11}, respectively.
$G_{(n)}$ decomposes the operator domain $\Omega_+\setminus \{\star\}$ of $M_{11}$ into two orbits of length $n$  and $11-n$.
Let  $\Omega_-$  be the complementary dodecad of  $\Omega_+$.
The following is immediate from the proof of Lemma~\ref{embM11}.
\begin{lem}\label{embM11(2)}
%$(2)$ 
Each $G_{(n)}$  decomposes $\Omega_-$ into two orbits.
Their length are  $\{6,6\}$ when $n=9,10$,  and  $\{2,10\}$ when $n=6$.
\end{lem}
%Let  $G_n \subset M_{11}$, $n=1, 2, 5$, be as in Lemma~\ref{embM11}.
We consider the orthogonal complement  of the invariant lattice for two actions  $G_{(n)} \act N_{\pm}$. % induced by Lemma~$\ref{embM11(2)}$.
The following is obvious.

\begin{lem}
Let  $G=G_{(n)} \act N_{\pm}$ $(n=6,9,10)$ be as above.

$(1)$ The orthogonal complement $N_{+, G}$ of the invariant lattice $N_+^G \subset N_+$  is the root lattice of type  $A_{n-1}+A_{10-n}$.

$(2)$ The orthogonal complement $N_{-, G}$ of the invariant lattice $N_-^G \subset N_-$  is a negative definite odd integral lattice of rank $10$.
$N_{-, G}$  contains an index-two sublattice which is of type $A_5 + A_5$ when $n=9,10$  and  $A_1 + A_9$ when $n=6$.
\end{lem}
  
The Niemeier lattice  $N$  contains  $N_{\pm}(2)$ as a primitive sublattice.
Since $N$  is unimodular we have an isomorphism
\begin{equation}\label{disc isom}
{\rm Disc}\, N_+(2) \simeq {\rm Disc}\, N_-(2)
\end{equation}
of discriminant groups.
This isomorphism is compatible with the actions of $M_{12}$.

Now we recall the modulo 2 reduction $l:= L/2L$ of an integral lattice  $L$.
When  $L$  is even, $l$  is endowed with the quadratic form
\begin{equation*}
q: l \to \Z/2\Z, \quad x \mapsto (x^2)/2
\end{equation*}
with value in  $\F_2 = \Z/2\Z$.
When  $L$  is odd, $l$  is endowed with the bilinear form
\begin{equation}
b: l \times l \to  \left(\frac12\Z\right) /\Z, \quad (x,  y) \mapsto (x. y)/2.
\end{equation}
The alternating part  $l^{alt} : = \{x \,|\, b(x, x) = 0\}$  of  $l$  is a subspace of codimension one, and carries 
\begin{equation}
q: l^{alt} \to \Z/2\Z, \quad x \mapsto (x^2)/2
\end{equation}
which is a {\it quadratic refinement} of  $b$, that is,
$q(x+y)-q(x)-q(y)=b(x,y)$  holds for every  $x, y \in l^{alt}$.

We need an Enriques counterpart of the isomorphism \eqref{disc isom}.
The $K3$ lattice  $\Lambda$  decomposes in two parts by the action of free involution  $\varepsilon$.
The invariant part is the Enriques lattice $\Lambda_+$ of type $T_{2,3,7}$ scaled by 2.
The anti-invariant part $\Lambda_-$, called the anti-Enriques lattice, is isomorphic to  $U + U(2) + E_8(2)$.
Since  $\Lambda$  and  $\Lambda_+$  are unimodular and  since $\Lambda$  contains the orthogonal direct sum $\Lambda_+(2) + \Lambda_-$, we have the isomorphism
\begin{equation}\label{Enriques counterpart}
\Lambda_+/2\Lambda_+ \simeq {\rm Disc}\, \Lambda_-
\end{equation}
of 10-dimensional quadratic spaces over  $\F_2$.

Returning to the action  $G=G_{(n)} \act N_{\pm}$, we put $L_\pm:= N_{\pm, G}$ and denote its modulo 2 reduction by  $l_\pm$.
Restricting the isomorphism \eqref{disc isom}  to  $l_+$ we have

\begin{lem}
Two $9$-dimensional quadratic spaces $(l_+, q_+)$ and $(l_-^{alt}, q_-)$  over  $\F_2$  are isomorphic to each other including their $G$-actions.
\end{lem}

\begin{rmk}
The bilinear form on $(l_+, q_+)$ has 1-dimensional radical, and $q_+$  takes value 1 at the nonzero element in the radical.
\end{rmk}

The lattice $N_G$, the orthogonal complement of $N^G \subset N$, is obtained by patching two lattices $L_{\pm}$  by the isomorphism in the lemma.
$N_G$  is an even lattice {\it of Leech type}, that is, the induced action of $G$  on the discriminant group  ${\rm Disc}\, N_G$  is trivial and $N_G$  does not have a $(-2)$ element.

As is observed in the introduction  $L_+$  have a primitive embedding  into the lattice of type  $T_{2,3,7}$.
The following is the counterpart of  $L_-$.

\begin{prop}\label{L_-}
The lattice $L_-(2)$  has a primitive embedding into the anti-Enriques lattice  $\Lambda_-$.
\end{prop}

The essential part is this.
\begin{lem}
$L_-$  has a primitive embedding into the odd unimodular lattice $I_{2,10} := \langle1\rangle^2 + \langle-1\rangle^{10}$ of signature  $(2, 10)$.
\end{lem}
\proof
We take  $\{h_1, h_2, e_1, \ldots, e_{10}\}$  with  $(h_1^2)= (h_2^2)= 1$ and $(e_i^2)= -1$ $(1 \le i \le 10)$ as an orthogonal basis of  $I_{2,10}$.

In the case  $n=6$,  $L_-$  is the unique (odd) integral lattice containing  $A_1+A_9$ as a sublattice of index 2.
$e_i-e_{i+1}$ ($i=1, \ldots, 9$) and $v=2(h_1+h_2)-\sum_1^{10}e_i$  generate a root sublattice of type  $A_9+A_1$ in $I_{2,10}$.
Since  the half sum of $e_{2j-1}-e_{2j} \in A_9$ ($j=1, \ldots, 5$)  and $v$ belongs to  $I_{2,10}$, the primitive hull of $A_9+A_1 \subset I_{2,10}$  is isomorphic to  $L_-$.

In the case $n=9, 10$, $L_-$  is the unique integral lattice containing  $A_5+A_5$ as a sublattice of index 2.
$L_-$  is isomorphic to the orthogonal complement of $2(h_1+h_2) - e_1 -e_2-e_3-e_4-e_5$ and $2(h_1-h_2) - e_6 -e_7-e_8-e_9-e_{10}$ in  $I_{2,10}$.
In fact, the orthogonal complement is generated by
\begin{gather}
e_1-e_2, e_2-e_3, e_3-e_4, e_4-e_5, h_1+h_2 - e_1 -e_2-e_3-e_4; \notag\\
e_6-e_7, e_7-e_8, e_8-e_9, e_9-e_{10}, h_1-h_2 - e_6 -e_7-e_8-e_9 \notag
\end{gather}
which form a root lattice  $R$ of type $A_5+A_5$, and $h_1-e_2-e_4-e_6-e_8 \not\in R$.
\qed

\begin{rmk}
In the above proof we make use of the fact that the blow-up of the projective space $\P^3$ at five points has a Cremona symmetry of type  $A_5$,
which is described by Dolgachev\cite{Do} in terms of root systems.
\end{rmk}

\medskip
{\it Proof of Proposition~$\ref{L_-}$.}
The anti-Enriques lattice  $\Lambda_-$  is obtained from  $I_{2,10}(2)$  by adding $(-2)$-element  $(h_1+h_2 - \sum_1^{10}e_i)/2$.
Since  $h_1+h_2 - \sum_1^{10}e_i$  does not belong to the modulo 2 reduction of  the image of  $L_- \hookrightarrow I_{2,10}$ constructed in the lemma, 
the induced embedding $L_-(2) \hookrightarrow \Lambda_-$  is also primitive.
\qed

\begin{rmk}
The above relation between the anti-Enriques lattice $\Lambda_-$ and $I_{2,10}$  is observed in Allcock\cite{A00}.
\end{rmk}

Patching together two primitive embeddings
$L_+(2) \hookrightarrow \Lambda_+(2)$, determined by $A_9,  A_4+A_5, A_1+A_8 \subset T_{2,3,7}$,
and $L_-(2) \hookrightarrow \Lambda_-$, we have the following.

\begin{prop}
There exist s a primitive embedding of  $N_G$  into the $K3$ lattice $\Lambda$  such that $N_G \cap \Lambda_+(2) = L_+(2)$ and $L \cap \Lambda_- = L_-(2)$.
\end{prop}

\medskip
{\it Proof of Theorem~$\ref{veryMathieu}$.}
Let  $G$ be one of the three groups $G_{(6)}, G_{(9)}, G_{(10)}$ (or equivalently $\mathfrak S_5, N_{72}, \mathfrak A_6)$.
Since  $N_G$  is of Leech type the action $G$ on $N_G$  extends to that on the $K3$ lattice.
By our construction it preserves $\Lambda_+(2)$  and  $\Lambda_-$.
Since $L_-(2)$ does not contain a $(-2)$-element, there exists an Enriques surface  $S_{(n)}$ ($n=6,9,10$)  such that  $H^{1,1}(S_{(n)}, \Z)^- \simeq L_-(2)$ by the subjectivity theorem (\cite{BPV}).
Let  $h \in H^2(S_{(n)}, \Z)_f$  be a primitive element perpendicular to  $L_+$.
$h$  is unique up to sign.
Replacing with $-h$ if necessary, we may assume that  $h$ belongs to the positive cone, that is, the connected component of  
$\{ x \in H^2(S, \R) \,|\, (x^2) > 0\}$
which contains ample classes.
There exists a composition $w \in {\rm O}(H^2(S_{(n)}, \Z)_f)$  of 
reflections with respect to smooth rational curves on  $S_{(n)}$  such that $w(h)$  is nef.
By the strong Torelli type theorem (\cite{BP}), the cohomological action of  $G \act H^2(S_{(n)}, \Z)_f$  twisted by  $w$ is realized by an algebraic action.
\qed

\begin{rmk}\label{N72&A6}
By construction and by Lemma~\ref{embM11(2)} and the Torelli type theorem,  two Enriques surfaces $S_{(9)}$  and  $S_{(10)}$  are isomorphic to each other.
The Enriques surface  $S_{(6)}$  is $\mathfrak S_5$-equivariantly isomorphic to that of  type VII in Kondo\cite{kondo86}.
In particular, ${\rm Aut}\, S_{(6)}$ is the symmetric group  $\mathfrak S_5$.
\end{rmk}

\begin{rmk}
A Mathieu action of  $G=C_2 \times \mathfrak A_4$ on an Enriques suface can be constructed lattice theoretically also.
Via the embedding $\mathfrak S_6 \hookrightarrow M_{12}$, $G$  is embedded into $M_{12}$ and decomposes $\Omega_\pm$  into three orbits of length 2, 4 and 6.
Hence the lattice $N_{\pm,G}$  contains the root lattice of type $A_1 + A_3+A_5$ as a sublattice of index two.
$N_{+,G}$ has a primitive embedding into the Enriques lattice  $\Lambda_+$ since $\Lambda_+$ contains the lattice  $T_{2,4,6}$  as a sublattice of index 2.
$N_{-,G}$ has a primitive embedding into  $I_{2,10}$  since  $N_{-,G} \simeq N_{+,G}$  and since $I_{2,10} \simeq \Lambda_+ + \langle1\rangle + \langle-1\rangle$.
Hence the same argument shows the existence of (a 1-dimensional family of) Enriques surfaces with Mathieu actions of $G$.
\end{rmk}

\section{K3 surface constructed in \cite{KOZ1, KOZ2}}\label{KOZconstruction}

In \cite{KOZ1, KOZ2} Keum, Oguiso and Zhang constructed a $K3$ surface with
an action by a group $\mathfrak{A}_6\splitext C_4$ and  
determined the abstract structure of the group. Here we show that it
contains a fixed point free involution $\varepsilon$ and the action by $\mathfrak{A}_6$
descends to $S=X/\varepsilon$.
Hence this gives another lattice theoretic construction of $\mathfrak{A}_6$-action.

We start with recalling their results.
\begin{thm}\label{KOZthm}
There exists a $K3$ surface $X$ with the following properties.
\begin{(enumerate)}
\item $X$ is a smooth $K3$ surface with Picard number $\rho =20$ and the transcendental lattice $T_X$ is given by the 
Gram matrix $\begin{pmatrix} 6 & 0 \\ 0 & 6 \end{pmatrix}$. 
\item $X$ is acted on by a group $\widetilde{\mathfrak{A}_6}=\mathfrak{A}_6\splitext C_4$. Here 
$\mathfrak{A}_6$ is the subgroup of symplectic automorphisms and satisfies 
$NS(X)^{\mathfrak{A}_6}=\mathbb{Z}H$, $(H^2)=20$. 
\item The image of the natural homomorphism $c\colon \widetilde{\mathfrak{A}_6}\rightarrow 
\mathrm{Aut}(\mathfrak{A}_6)$ is $M_{10}$.
%\item By the conjugation above, 
%the group $\tilde{\mathfrak{A}}_6$ is isomorphic to the subgroup of $M_{10}\times C_4$ 
%generated by $\mathfrak{A}_6\subset M_{10}$ and $\tilde{g}=(g,\zeta_4)$, 
%where $g$ is an element in $M_{10}-\mathfrak{A}_6$ of order $4$.  
\end{(enumerate)}
\end{thm}
We put $\varepsilon$ 
to be the nontrivial element in $\ker c$ (namely the element $(1,-1)\in M_{10}\times C_4$
in the notation of \cite[Theorem 2.3]{KOZ2}).
The corresponding automorphism on $X$ is denoted by the same letter.
\begin{lem}
The automorphism $\varepsilon$ is fixed point free.
\end{lem}
\begin{proof}
By the construction, $\varepsilon$ is a non-symplectic involution on $X$, hence 
its fixed locus is a disjoint union of smooth curves. Assume it is not empty. We look at
the divisor $D$ given by the sum of fixed curves. Since $\varepsilon$ commutes with 
$\mathfrak{A}_6$, $D$ belongs to 
the sublattice $NS(X)^{\mathfrak{A}_6}=\mathbb{Z}H$ by Theorem \ref{KOZthm} (2). 
Since $H$ is ample, $D$ is connected.
Then $(H^2)=20$ shows that the genus of $D$ is (at least) $11$, but there are no 
such fixed curves for non-symplectic involutions.
Thus $\varepsilon$ is free.
\end{proof}
Therefore, the Enriques surface $S=X/\varepsilon$ has an action by $\mathfrak{A}_6$
(or by $M_{10}$, more precisely).

%%%%%%%%%%%%%%%%%%%%%%%%%%%%%%%%%%%%%%%%%%%%%%%%

\end{document}